\documentclass[11pt]{amsart}
\usepackage[margin=3cm]{geometry}
\usepackage{mathrsfs}
\usepackage{amsmath,amsfonts,amssymb,amsthm,mathscinet}
\usepackage{bbm}
\usepackage{epstopdf}
\usepackage{graphicx}
\usepackage[usenames,dvipsnames,svgnames,table]{xcolor}
\usepackage[linktocpage=true,colorlinks=true,linkcolor=blue,citecolor=red,urlcolor=magenta,pdfborder={0 0 0}]{hyperref}

\theoremstyle{plain}
\newtheorem{theorem}{Theorem}[section]
\newtheorem{corollary}[theorem]{Corollary}
\newtheorem{proposition}[theorem]{Proposition}
\newtheorem{lemma}[theorem]{Lemma}
\theoremstyle{definition}
\newtheorem{definition}[theorem]{Definition}
\newtheorem{remark}[theorem]{Remark}

\numberwithin{equation}{section}

\newcommand*{\dif}{\mathop{}\!\mathrm{d}}
\newcommand{\trace}{\mathop{\mathrm{tr}}}
\newcommand{\R}{\mathbb{R}}
\newcommand{\T}{\mathbb{T}}
\newcommand{\N}{\mathbb{N}}
\newcommand{\RO}{\mathcal{R}}
\newcommand{\UO}{\mathcal{U}}

\def \Z {{\mathbb{Z}}}
\def \H {\mathcal{H}}
\def \E {\mathcal{E}}
\def \EE {\mathscr{E}}
\def \L {\mathscr{L}}
\def \LFP {\mathscr{L}_{\scriptscriptstyle{\rm FP}}}
\def \lfp {\mathscr{L}_{\scriptscriptstyle{\rm OU}}}
\def \dk {{\rm deg}_{\rm kin}}
\def \Ck {{\mathcal C}_l}
\def \Ckin {{\mathcal C}_{\rm kin}}
\def \rfp {\mathscr{R}}
\def \fin {f_{\rm in}}
\def \hin {h_{\rm in}}
\def \hine {h_{\e,\rm in}}
\def \K {\mathcal{K}}

\def \A {\mathcal{A}}

\def \Y {\mathcal{Y}}

\def \gg {\tilde{g}}

\def \p{\partial}

\def \trace {{\text{\rm tr}}}

\def \o {{\omega}}
\def \a {{\alpha}}
\def \b {{\beta}}

\def \g {{\gamma}}
\def \d {{\delta}}
\def \e {{\epsilon}}

\def \r {{\rho}}
\def \s {{\sigma}}
\def \t {{\tau}}
\def \m {{\mu}}
\def \n {{\nu}}
\def \x {{\xi}}
\def \y {{\eta}}
\def \w {{\omega}}
\def \G {{\Gamma}}
\def \O {{\Omega}}

\newcommand{\la}{\langle}
\newcommand{\ra}{\rangle}

\begin{document}
	\title[On the nonlinear kinetic Fokker-Planck equation]{On a spatially inhomogeneous nonlinear Fokker-Planck equation: Cauchy problem and diffusion asymptotics}	
	\date{\today}
	\author{Francesca Anceschi}
	\address{Dipartimento di Ingegneria Industriale e Scienze Matematiche, Università Politecnica delle Marche,
		       Via Brecce Bianche 12, 60131 Ancona, Italy}
	\email{f.anceschi@staff.univpm.it}
	\author{Yuzhe Zhu}
	\address{D\'epartement de math\'ematiques et applications, \'Ecole normale sup\'erieure, 45 Rue d'Ulm, 75005 Paris, France}
	\email{yuzhe.zhu@ens.fr}
	
	
\begin{abstract}
We investigate the Cauchy problem and the diffusion asymptotics for a spatially inhomogeneous kinetic model associated to a nonlinear Fokker-Planck operator. 
We derive the global well-posedness result with instantaneous
smoothness effect, when the initial data lies below a Maxwellian. 
The proof relies on the hypoelliptic analog of classical parabolic theory, as well as a positivity-spreading result based on the Harnack inequality and barrier function methods. 
Moreover, the scaled equation leads to the fast diffusion flow under the low field limit. 
The relative phi-entropy method enables us to see the connection between the overdamped dynamics of the nonlinearly coupled kinetic model and the correlated fast diffusion. 
The global in time quantitative diffusion asymptotics is then derived by combining entropic hypocoercivity, relative phi-entropy and barrier function methods. 
\end{abstract}
\maketitle

\hypersetup{bookmarksdepth=2}
\setcounter{tocdepth}{1}
\tableofcontents
	
\section{Introduction}\label{intro}
We consider the kinetic Fokker-Planck operator $\LFP:=\nabla_v\cdot\left(\nabla_{v}+v\right)$ and the following spatially inhomogeneous nonlinear drift-diffusion model, 
\begin{equation}\label{nk}
\left\{ 
\begin{aligned}
\ &\left(\p_t+v\cdot\nabla_x\right)f(t,x,v)=\rho^\b_f(t,x)\,\LFP f(t,x,v),  \\
\ &\;f(0,x,v)=\fin (x,v), \\
\end{aligned}
\right. 
\end{equation} 
for an unknown $f(t,x,v)\ge0$ with $t\in\R_+$, $x\in\T^d$ or $\R^d$, $v\in\R^d$, where $\T^d$ denotes the $d$-dimensional torus with unit volume, the constant $
\b\in[0,1]$ and 
\begin{equation*}
\rho_f(t,x) := \int_{\R^{d}} f(t,x,v) \dif v. 
\end{equation*}
Given a constant $\e\in(0,1)$, the equation under the low field scaling $t\mapsto\e^2 t$, $x\mapsto\e x$ reads 
\begin{equation}\label{epfnk}
\left\{ 
\begin{aligned}
\ &\left(\e\p_t+v\cdot\nabla_x\right)f_\e(t,x,v)=\frac{1}{\e}\rho^\b_{f_\e}(t,x)\,\LFP f_\e(t,x,v),  \\
\ &\;f_\e(0,x,v) = f_{\e,\rm in} (x,v). \\
\end{aligned}
\right. 
\end{equation} 
Our aim is to show the global well-posedness and the trend to equilibrium with smoothness a priori estimates for the equation~\eqref{nk}, and the quantitative asymptotic dynamics of the equation~\eqref{epfnk} as $\e$ tends to zero. 

\subsection{Main results}
Let us recall that a classical solution of an evolution equation is a nonnegative function verifying the equation pointwise everywhere and 
matching the initial data continuously.
Unless otherwise specified, any solution we consider below is intended in the classical sense. 
For $k\in\N$, $\mathcal{C}^k(\O)$ is the set of functions having all derivatives of order less than or equal to $k$ continuous in the domain $\O$. 
For $\a\in(0,1)$, $\mathcal{C}^\a(\O)$ is the classical H\"older space on $\O$ with exponent $\a$. 
Besides, we write the measure $\dif m:=\dif x\dif\m$, where
\begin{equation*}
 \m(v):=(2\pi)^{-\frac{d}{2}}e^{-\frac{|v|^2}{2}}  {\quad\rm and\quad}   
\dif\m:=\m\dif v 
\end{equation*}
denote the Gaussian function and the Gaussian measure, respectively. 
A function takes the form of $C\m(v)$ for some constant $C>0$ is called a Maxwellian.

\begin{theorem}\label{wellpose}
Let the space domain $\O_x=\T^d$ or $\R^d$, and the constants $0<\lambda<\Lambda$ be given. 
\begin{item}
\item [\bf {\ \,(i)}] 
If $\fin\in{\mathcal C}^0(\O_x\times\R^d)$ satisfies $0\le\fin \le \Lambda \m$ in $\O_x\times\R^d$, then there exists a solution $f$ to the Cauchy problem~\eqref{nk} such that $0\le f \le \Lambda \m$ in $\R_+\times\O_x\times\R^d$. 
Moreover, for any $\n\in(0,1)$, $k\in\N$, and any compact subset $K\subset(0,T]\times\O_x$, there is some constant $C_{T,\n,k,K}>0$ depending only on $d,\b,\lambda,\Lambda,T,\n,k,K$ and the initial data such that
\begin{equation*}\label{golbalregu}
\|\m^{-\n} f\|_{{\mathcal C}^{k}(K\times\R^d)} \le C_{T,\n,k,K}.  
\end{equation*}
Additionally, if $\fin$ is H\"older continuous and $\rho_{\fin}\ge\lambda$ in $\O_x$, then the solution that lies below any Maxwellian is unique. 
			
\medskip
\item [\bf {\;(ii)}]
For $\O_x=\T^d$, if the initial data satisfies $\lambda\m\le\fin\le \Lambda\m$ in $\T^d\times\R^d$, 
then for any $k\in\N$, there exists some constant $c>0$ depending only on $d,\b,\lambda,\Lambda$ and some constant $C_k>0$ depending additionally on $k$ such that for any $t\ge1$,
\begin{equation*}\label{longtimethm}
\left\|\frac{f-\m\int \fin\dif x\dif v}{\sqrt{\m}}\right\|_{{\mathcal C}^k(\mathbb{T}^{d}\times\R^d)} 
\le C_k e^{-ct}. 
\end{equation*}
For $\O_x=\R^d$, if the initial data satisfies $\lambda\m\le\fin\le \Lambda\m$ in $\R^d\times\R^d$ and $\fin-M_1\m\in L^1(\R^d\times\R^d)$ for some constant $M_1>0$, then there is some constant $C'>0$ depending only on $d,\b,\lambda,\Lambda,M_1$ such that
\begin{equation*}
\left\|\frac{f-M_1\m}{\sqrt{\m}}\right\|_{L^2(\R^d\times\R^d)}
\le C'\left(1+\|\fin-M_1\m\|_{L^1(\R^d\times\R^d)}\right) t^{-\frac{d}{4}}. 
\end{equation*}
\end{item}
\end{theorem}

\begin{remark}\label{wellposer}
For general measurable initial data $f_{\rm in}$, if it satisfies $f_{\rm in} \le \Lambda \mu$ and an extra locally uniform lower bound assumption (see \eqref{lowerdd} below for a precise description), the existence of solutions still holds in some weak sense, as pointed out in Remark~\ref{remark-weak} below. 
\end{remark}

In order to describe the diffusion asymptotics of the equation~\eqref{epfnk}, we introduce the (Bregman) distance characterized by the relative phi-entropy functional $\H_\b$. 
\begin{definition}\label{phi-entropy}
Let $\b\in[0,1]$. For any measurable functions $h_1\ge0$ and $h_2>0$ defined in $\T^d\times\R^d$, the \emph{relative phi-entropy} of $h_1$ with respect to $h_2$ is defined by
\begin{align*}
\H_\b(h_1|h_2):=&
\int_{\T^d\times\R^d}\left( \varphi_\b(h_1)-\varphi_\b(h_2)-\varphi'_\b(h_2)(h_1-h_2) \right)\dif m, 
\end{align*}
where $\varphi_\b:\R_+\rightarrow\R$ is defined by $\varphi_\b(z):=\frac{1}{1-\b}\left(z^{2-\b}-(2-\b)z+1-\b\right)$ for $\b\in[0,1)$ and $\varphi_1(z):=z\log z-z+1$. 
\end{definition}
	
\begin{theorem}\label{limit}
Let the constants $\a_0\in(0,1)$, $0<\lambda<\Lambda$ be given and consider a sequence of functions  $\left\{f_{\e,\rm in}\right\}_{\e\in(0,1)}\subset 
{\mathcal C}^{\a_0}(\T^d\times\R^d)$ satisfying $0\le f_{\e,\rm in} \le \Lambda\m$ in $\T^d\times\R^d$ and $\r_{f_{\e,\rm in}}\ge\lambda$ in $\T^d$. 
Let $f_\e$ be the solution to \eqref{epfnk} associated with the initial data $f_{\e, \rm in}$. 
\begin{item}
\item [\bf {\ \,(i)}]
If there exists some constant $\e'\in(0,1)$ and some function $\rho_{\rm in}\in {\mathcal C}^{\a_0}(\T^d)$ valued in $[\lambda,\Lambda]$ such that
\begin{equation*}
	\H_\b\left(\,\m^{-1}f_{\e,\rm in}\,|\,\rho_{\rm in}\,\right)\le \e', 
\end{equation*}
then there exist some constants $M,m>0$ depending only on $d,\b,\lambda,\Lambda,\a_0, \left\|\r_{\rm in}\right\|_{{\mathcal C}^{\a_0}(\T^d)}$ and $\left\|f_{\e,\rm in}\right\|_{{\mathcal C}^{\a_0}(\T^d\times\R^d)}$ such that for any $T>0$, 
\begin{equation*}\label{limit1}
\left\|\m^{-1}f_\e-\rho\right\|_{L^\infty\left([0,T];\,L^2\left(\T^d\times\R^d,\,\dif m\right)\right)}
\le Me^{MT}(\e+\e')^{m},  
\end{equation*}
where $\rho(t,x)$ for $(t,x)\in\R_+\times\T^{d}$ is the solution to the following fast diffusion equation, 
\begin{equation}\label{limiteq}
\left\{ 
\begin{aligned}
\ &\partial_t\rho(t,x)=\nabla_x\cdot\left(\rho^{-\b}(t,x)\nabla_x\rho(t,x)\right), \\
\ &\;\rho(0,x) = \rho_{\rm in} (x).  \\
\end{aligned}
\right. 
\end{equation}

\medskip
\item [\bf {\;(ii)}]
If we additionally assume that $f_{\e,\rm in}\ge\lambda\m$ in $\T^d\times\R^d$, 
then there exist some constants $M',m'>0$ with the same dependence as $M,m$ such that 
\begin{equation*}\label{limit2}
\left\|\m^{-1}f_\e-\rho\right\|_{L^\infty\left(\R_+;\,L^2\left(\T^d\times\R^d,\,\dif m\right)\right)}
\le M'(\e+\e')^{m'}. 
\end{equation*}
\end{item}
\end{theorem}

\subsection{Strategy and background}\label{background}

\subsubsection{Cauchy problem of the nonlinear model}
The study of the well-posedness of the nonlinear model~\eqref{nk} was firstly addressed
in \cite{IM} mixing H\"older and Sobolev spaces in the torus, and in \cite{LWYang} under the regime of perturbation to the 
global equilibrium in the whole space. 
We develop it with rough initial data by means of the combination of hypoelliptic analog of the parabolic theory with a positivity-spreading result; 
in particular, the technique we employ allows us to drop the smallness and lower bound assumptions as asserted in Theorem~\ref{wellpose}. 
Besides, the global behavior of solutions to \eqref{nk} is derived under the assumption of upper and lower bounds on the initial data only. 

When the drift-diffusion coefficient $\rho^\b_f$ in \eqref{nk} is proportional to the local mass of the solution, that is when $\b=1$, the equations \eqref{nk}, \eqref{epfnk} have the same quadratic homogeneity as the Landau equation, but simpler global bounds and conservation laws. 
Due to the complex structure of the Landau equation, most of the existing results for its classical solutions are about the global theory under the near Maxwellian equilibrium regime \cite{Guo,KGH}, and about the local well-posedness associated with low-regularity and non-perturbative initial data \cite{HST2,HST1}. 
By contrast, the boundedness from above and from below by Maxwellians of the initial data will be preserved along 
time for the solutions to \eqref{nk}, \eqref{epfnk}, 
and the lack of conservation of momentum and energy of \eqref{epfnk} reduces its hydrodynamic limit to the fast diffusion 
flow~\eqref{limiteq} rather than the Navier-Stokes dynamics of scaling limit of the Landau equation, 
which makes its Cauchy problem and global behavior more tractable in a very general setting.

The method we propose to address the nonlinear Cauchy problem with the only requirement of the initial data lying below a Maxwellian involves several ingredients. 
First of all, in Section~\ref{linearsection} we carry out a preliminary study on the linear counterpart of \eqref{nk}, that is the Cauchy problem associated to the Kolmogorov operator, 
\begin{equation}\label{k}
\L_1:=\partial_t+v\cdot\nabla_x-\trace\left(A(t,x,v)D_v^{2}\cdot\right)+B(t,x,v)\cdot\nabla_v,
\end{equation} 
where the coefficients including the entries of the positive definite $d\times d$ real symmetric matrix $A$ and the 
$d$-dimensional vector $B$ are H\"older continuous ($B$ is not necessarily bounded over $v\in\R^d$). 
Even if the well-posedness theory for the Cauchy problem associated to the linear operator~\eqref{k} was already 
well-developed in some sense in the existing literature (see \cite{M}, as well as the survey paper \cite{APsurvey} and the references therein), 
the H\"older spaces (see Definition~\ref{kineticholder}) considered in these works are instead different from the one studied in \cite{ISkinetic,IM} (see Definition~\ref{holdercontinuous}) which we will use. 
Indeed, in contrast to \cite{IM}, the (Schauder type) apriori estimates proved in the previous literature are weaker and not appropriate to bootstrap higher regularity for nonlinear problems (see Subsection~\ref{smoothsection}).

Secondly, the treatment of the existence issue for \eqref{nk} in H\"older spaces is based on a fixed point argument, where the compactness is provided by hypoelliptic regularization results; see Subsection~\ref{wellsection}. 
A breakthrough on such a priori estimates for spatially inhomogeneous kinetic equations with a quasilinear diffusive structure in velocity was obtained in the works \cite{GIMV} and \cite{HS,IM}, where the authors proved the kinetic (hypoelliptic) counterparts of the De Giorgi-Nash-Moser theory and the Schauder theory for classical elliptic equations (see for instance \cite{GilbargTrudinger}), respectively. One may refer to \cite{Mouhot} for a summary. 
Armed with the Schauder estimate developed in \cite{IM} in kinetic H\"older spaces and the bootstrap procedure developed in \cite{IS} adapted to our case, we are then able to derive instantaneous ${\mathcal C}^\infty$ regularization for the solutions to the equation~\eqref{nk} in Subsection~\ref{smoothsection}, provided that the solution is bounded from above and bounded away from vacuum, which guarantees the ellipticity in the velocity variable for \eqref{nk}.

Thirdly, in order to remove the lower bound assumption on the initial data, in Subsection~\ref{lowersection} we establish a self-generating lower bound result showing that the positivity of solutions spreads everywhere instantaneously. 
Its proof is based on repeated applications of the spreading of positivity forward in time (see Lemma~\ref{lowerlemma1}) and the spreading for all velocities (see Lemma~\ref{lowerlemma2}), as proposed in \cite{HST3}. 
On the one hand, the barrier function argument will be used in the same spirit as \cite{HST3} to show the former one. 
Indeed, a \emph{lower (resp.\! upper) barrier} for a certain equation is a  subsolution (resp.\! supersolution) of the equation which bounds its solution from below (resp.\! above) on the boundary; it then follows from the maximum principle that the barrier function performs as a lower (resp.\! upper) bound of the solution. 
On  the  other  hand, combining the local Harnack inequality obtained in \cite{GIMV} with the construction of a Harnack chain yields the latter one. We remark that the idea of the Harnack chain was firstly used in \cite{Moser} and an example of its application to Kolmogorov equations can be found in \cite{AEP}. 
Essentially, the spreading of positivity can be seen as a lower bound estimate of the fundamental solution, which is thus related to the result in \cite{HST2}, where they applied a probabilistic method. 

A subtle point of the lower bound result lies in the possibilities of the degeneracy of solutions as $t\rightarrow0^+$ or 
$t\rightarrow\infty$, which leads to two delicate issues. 
First, with the same difficulty as mentioned in \cite{HST1}, in order to prove the uniqueness of the Cauchy problem~\eqref{nk}
the nondegeneracy of diffusion up to the initial time is required so that the a priori estimates can be still applicable. 
We remark that, generally speaking, deriving uniqueness of solutions to nonlinear equations in rough spaces is always a classical difficulty, and the presence of a vacuum sometimes gives rise to non-uniqueness phenomenon even for the limiting equation \eqref{limiteq} (see for instance \cite{DK}). 
Under the additional assumptions of absence of vacuum on the local mass and H\"older continuity of initial data, we achieve the uniqueness by using the scaling argument and Grönwall's lemma, since the H\"older estimate around the initial time implies that the integrand in the inequality of Grönwall's type is improved to be integrable with respect to the time variable; see the proof of Theorem~\ref{unique} below for more details. 
Second, we are only able to show the convergence to equilibrium if the drift-diffusion coefficient $\rho^\b_f$ decays slower than $t^{-1}$ as $t\rightarrow\infty$ in Proposition~\ref{longtime}. Therefore, an additional lower Maxwellian bound on the initial data is imposed in Theorem~\ref{wellpose}~(ii) and in Theorem~\ref{limit}~(ii) to ensure the solutions will be away from the vacuum uniformly along time. 
It would be expected that such additional lower bound assumption could be removed, especially when $\b$ is small.

\subsubsection{Long time behavior}
The drift-diffusion operator $\LFP$ acts only on the velocity variable and ceases to be dissipative on its unique steady state $\m$, which also ensures that the null space of $\LFP$ is spanned by $\m$ and the conservation law of mass is satisfied. Consequently, the convergence to equilibrium is to be expected.  
With the help of the global smoothness a priori estimates, we are able to pass from the exponential convergence to equilibrium in $L^2$-framework to the uniform convergence in ${\mathcal C}^\infty$ in Subsection~\ref{longtimesection}, when the spatial domain is compact, that is the periodic box $\T^d$. 
Therein, the $L^2$-convergence is obtained by the $L^2$-hypocoercivity under a macro-micro (fluid-kinetic) decomposition scheme, which suggests to construct some proper entropy (Lyapunov) functional that would provide an equivalent $L^2$-norm for solutions. 
The key ingredient is to control the macroscopic part by means of the microscopic part in view of the decomposition. 
This hypocoercive theory was studied in \cite{EGKM,DMS,He} via different approaches, while their ideas are essentially the same. 
In \cite{EGKM}, the authors intended to develop the nonlinear energy estimate in an $L^2$-$L^\infty$ framework. 
In \cite{DMS} and \cite{He}, the authors studied the $L^2$-hypocoercivity theory in an abstract setting and in the framework of pseudo-differential calculus, respectively.  
Besides, if the spatial domain is $\R^d$, meaning that it is not confining in a compact region, then the convergence rate slows down to an algebraic decay, whose hypocoercive theory was captured by \cite{BDM}. We remark that the $L^2$-framework allows us to avoid some difficulties from the nonlinearity of the operator $\rho^\b_f\LFP f$, in contrast with $H^1$-entropic hypocoercivity methods (see for instance the memoir~\cite{Villanihypo}).

\subsubsection{Diffusion asymptotics}
The diffusion approximation serves as a simplification of collisional kinetic equations when the mean free path is much smaller than the typical length of 
observation in a long time scale. This approximation for linear Fokker-Planck models can be traced back to \cite{Degond}, where the authors applied the Hilbert expansion method. 
One is also able to achieve the diffusion limit for \eqref{epfnk} in some weak sense by applying a similar strategy to the one given in \cite{ELM}. However, weak convergence is sometimes ineffective for application, as precise description of the convergence is not given. 
Still the nonlinearity of the term $\rho^\b_{f_\e}\LFP f_\e$ in \eqref{epfnk} associated with non-perturbative initial data reveals some difficulties to derive a quantitative convergence. 

In order to overcome this difficulty, in Subsection~\ref{finitetimesection} we will rely on the phi-entropy of solutions relative to their limit to see the finite time asymptotics on the torus.
The relative entropy method, that heavily relies on the regularity of solutions to the target equation,  has become an effective tool in the study of hydrodynamic limits since \cite{YauHZ,BGL} (see also \cite{LSR}). 
The method applied to the diffusion asymptotics of the kinetic Fokker-Planck equation of the type with linear diffusion can be found in \cite{Markou}. 
The so-called phi-entropy (relative to the global equilibrium) was used to study the convergence of certain kinds of Fokker-Planck equations; see for instance \cite{AMTU,DL}. 
Finally, combining the barrier function method with a careful treatment of the regularity estimate of the target equation enables us to deal with the asymptotic dynamics for the cases associated with general H\"older continuous initial data.

\subsection{Physical motivation}
The spatially inhomogeneous Fokker-Planck equation~\eqref{nk} arises from modeling the evolution of some system of a large number of interacting particles from the statistical mechanical point of view. 
These models appear for instance in the study of plasma physics and biological dynamics; see \cite{Villani,Cha2}. 
Its solution can be interpreted as the probability density of the particles lying at the position $x$ at time $t$ with velocity $v$. The scaled model \eqref{epfnk} for small $\e$ describes the evolution of the particle density in the small mean free path and long time regime, where the nondimensional parameter $\e\in(0,1)$ designates the ratio between the mean free path (microscopic scale) and the typical macroscopic length. 
The limiting equation~\eqref{limiteq} characterizes its macroscopic dynamics. 

From the perspective of a stochastic process $\{(X_t,V_t):t\ge0\}$ driven by a Brownian motion $\{\mathcal{B}_t:t\ge0\}$, 
\begin{equation*}
\left\{ 
\begin{aligned}
\ &\dif X_t= V_t\dif t, \\   
\ &\dif V_t= \r^\b_f(t,X_t)V_t\dif t+\sqrt{2\r^\b_f(t,X_t)}\dif\mathcal{B}_t, \\
\end{aligned}
\right. 
\end{equation*} 
the dual equation describing the dynamics of $\{(X_t,V_t):t\ge0\}$ is given by \eqref{nk}; see the review paper \cite{Chand}. 
Indeed, the nonlinear term $\rho^\b_f\LFP f$ models the collisional interaction of the particles, where the mobility of these particles is hampered by their aggregation. 
More precisely, the nonlinear dependence on the drift-diffusion coefficient $\rho^\b_f$ translates the fact that the friction effect in the interaction is positively correlated to the local mass of particles occupying the position $x$ at time $t$. 
Moreover, the low field scaling $t\mapsto\e^2 t$, $x\mapsto\e x$ of the equation~\eqref{nk} formally implies the equation~\eqref{epfnk}. As $\e$ tends to zero, its spatial diffusion phenomena are characterized by the equation~\eqref{limiteq}. 

Regarding its physical interpretation, we point out that the factor multiplying the time derivative in \eqref{epfnk} takes into account the long time scale. The inverse of the factor multiplying $\rho^\b_{f_\e}\LFP f_\e$ stands for the scaled average distance traveled by particles between each collision and it is usually referred to as mean free path. 
In the small mean free path regime, it was noticed in \cite{Chand} that the spatial variation occurs significantly only under the long time scale that is consistent with the particle motion. 
In such an overdamped process, named as low field limit or diffusion limit, the kinetic model is reduced to a macroscopic model.

Finally, we recall that the associated phi-entropy introduced in Definition~\ref{phi-entropy} is also known as Tsallis entropy in the physics community, which generalizes the Boltzmann–Gibbs entropy (the  phi-entropy with $\beta=1$) in nonextensive statistical mechanics~\cite{tsallis1}.  
It gives some hints for the formulation of the correlated diffusion, where the index $\beta$ measures the degree of nonextensivity and nonlocality of the system; see \cite{tsallisbook} and the references therein.

\subsection{Organization of the paper}
The article is organized as follows. In Section~\ref{preliminaries}, we recall some basic notions related to kinetic H\"older spaces that are adapted to the Fokker-Planck equations. 
Section~\ref{linearsection} is devoted to the study of the linear Fokker-Planck equation with H\"older continuous coefficients. 
The well-posedness result Theorem~\ref{wellpose}~(i) is proved in Section~\ref{nonlinear}. 
The asymptotic behaviors, including Theorem~\ref{wellpose}~(ii) and Theorem~\ref{limit}, are proved in Section~\ref{diffusionlimit}. 
	
\medskip\noindent\textbf{Acknowledgement.}
The authors are grateful to Cyril Imbert for suggesting the question, and both François Golse and Cyril Imbert for the helpful discussions, and the anonymous referees for their careful reading and comments.
\includegraphics[height=0.275cm]{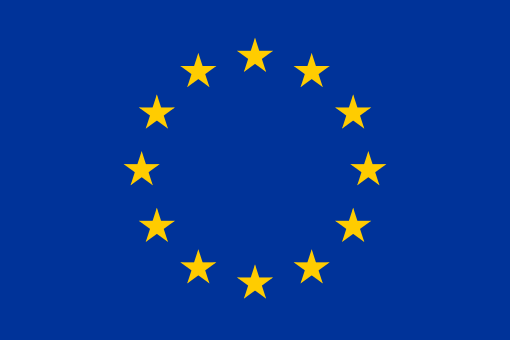} 
YZ's research has received funding from the European Union’s Horizon 2020 research and innovation programme under the Marie Skłodowska-Curie grant agreement No 754362.

\section{Preliminaries}\label{preliminaries}
This section is devoted to the basic notations, including the invariant structure and the kinetic H\"older space for the equations we are concerned with. 
Instead of the usual parabolic scaling and translations, the  invariant scaling and transformation associated with the Kolmogorov operator $\L_1$ (see \eqref{k}) is replaced by kinetic scaling and Galilean transformations, respectively. 
It then turns out that the appropriate H\"older space as well as its norm should be adapted to the new scaling and transformation. 
Let us now precisely state these notions.

\subsection{The geometry associated to Kolmogorov operators}
Let $z:=(t,x,v)\in\R\times\R^d\times\R^d$. 
We define the \emph{kinetic scaling}, 
\begin{equation*}
S_r(t,x,v) := (r^2 t,\, r^{3} x,\, r v), {\quad\rm for\ } r>0,  
\end{equation*}
and the \emph{Galilean transformation}, 
\begin{equation*}
(t_0,x_0,v_0)\circ(t,x,v):=(t_0+t,\,x_0+x+tv_0,\,v_0+v), 
	{\quad\rm for\ } (t_0,x_0,v_0)\in\R\times\R^d\times\R^d.
\end{equation*}
With respect to the product $\circ$, we are able to define the inverse of $z$ as given by $z^{-1}:=(-t,-x+tv,-v)$. 
In view of this structure of scaling and transformation, it is natural to define the cylinder centered at the origin of
radius $r>0$ as 
$$Q_r:= (-r^2, 0] \times B_{r^3}(0) \times B_r(0).$$
More generally, the cylinder  centered at $z_0=(t_0,x_0,v_0)$ with radius $r$ is defined by 
\begin{align*}
Q_r(z_{0}):=&\left\{ z_{0} \circ S_r(z): \, z \in Q_1\right\}\\
=&\left\{ (t,x,v): \, t_{0} - r^{2} < t \le t_{0}, \, |x - x_{0} - (t - t_{0})v_{0}| < r^{3},\, |v - v_{0}| < r\right\}.
\end{align*}
Roughly speaking, for fixed $z_0\in\R^{1+2d}$, the Kolmogorov operator $\L_1$ is invariant under the kinetic scaling and left invariant under the Galilean transformation. 
It means that if $f$ is a solution to the equation $\L_1f=0$ in $Q_{r}(z_{0})$, then $f\left(z_0\circ S_r(\cdot)\right)$ solves an equation of the same ellipticity class in $Q_1$. 

Besides, the associated quasi-norm $\|\cdot\|$ is defined by 
\begin{align*}
\|z\|:=\max\Big\{ |t|^{\tfrac12}, |x|^{\tfrac13}, |v| \Big\}, 
\end{align*}   
as we notice that $\|S_r(z)\|=r\|z\|$ and $\|z_0\circ z\|\le3\left(\|z_0\|+\|z\|\right)$. 
For further information on the non-Euclidean geometry associated to Kolmogorov operators, one may refer to \cite{APsurvey, ISkinetic} and the references therein. 

\subsection{Kinetic H\"older spaces and differential operators}
The proper kinetic H\"older space and kinetic degree of basic differential operators should be adapted to the above definitions, so that they are homogeneous under these transformations. Their definitions were introduced in \cite{ISkinetic} (see also \cite{IM}), and recalled below. 

Given a monomial $m (t,x,v) =t^{k_0} x_1^{k_1} \ldots x_d^{k_d} v_{1}^{k_{d+1}} \ldots v_d^{k_{2d}}$, we define its \emph{kinetic degree} as 
\begin{equation*}
\dk(m)= 2 k_{0} + 3 \sum\nolimits_{j=1}^{d} k_{j} + \sum\nolimits_{j=d+1}^{2d} k_{j}.
\end{equation*}
Any polynomial $p\in\R[t,x,v]$ can be uniquely written as a linear combination of monomials and its kinetic degree $\dk(p)$ is defined by the maximal kinetic degree of the monomials appearing in $p$. 
This definition is justified by the fact that $p(S_r(z)) = r^{\dk(p)} p(z)$.

\begin{definition}\label{holdercontinuous}
Let the constant $\a>0$ and the open subset $\O\subset\R\times\R^d\times\R^d$ be given. We say a function
$f : \O \rightarrow \R$ is $\Ck^\a$-continuous at a point $z_0\in\O$, 
if there exists a polynomial $p_{0}\in \R[t,x,v]$ with $\dk (p_{0})<\a$ and a constant $C>0$ such that, for any $z\in\O$ with $z_0\circ z\in\O$,  
\begin{equation}\label{holdercontinuousineq}
| f(z_0\circ z) - p_{0}(z) | \le C \| z\|^\a. 
\end{equation}
If this property holds for any $z_0,z$ on each compact subset of $\O$, then we say $f \in \Ck^{\a}(\O)$. 
If the constant $C$ in \eqref{holdercontinuousineq} is uniformly bounded for any $z_0,z\in\O$, 
we define the smallest value of $C$ as the semi-norm $[f]_{\Ck^{\a}(\O)}$, 
and the norm $\|f\|_{\Ck^{\a}(\O)}:=[f]_{\Ck^0(\O)}+[f]_{\Ck^{\a}(\O)}$, where we additionally define $\Ck^0(\O):={\mathcal C}^0(\O)$, the space of continuous functions on $\O$, with the norm $\|f\|_{\Ck^0(\O)}:=[f]_{\Ck^0(\O)}:=\|f\|_{{\mathcal C}^0(\O)}=\|f\|_{L^\infty(\O)}$. 
\end{definition}
\begin{remark}
For $\a\in[0,1)$, this $\Ck^\a$-continuity is equivalent to the standard definition of 
${\mathcal C}^\a$-continuity with respect to the distance $\|\cdot\|$.
The subscript ``\textit{l}'' of $\Ck$ stems from the definition of H\"older continuity above, that it is given in terms of a  left-invariant distance with respect the group structure of $\circ$. 
\end{remark}

We also mention another kind of H\"older space suitable for the study of Kolmogorov operators that was first used in \cite{M}. 
\begin{definition}\label{kineticholder}
Let $\a\in[0,1)$ and $\O\subset\R\times\R^d\times\R^d$ be given. 
The space $\Ckin^{2+\a}(\O)$ consists of functions $f\in\Ck^0(\O)$ such that $D^2_vf$, $(\p_t+v\cdot\nabla_x)f\in\Ck^\a(\O)$, equipped with the norm 
 \begin{equation*}
 \|f\|_{\Ckin^{2+\a}(\O)}:=\|f\|_{\Ck^0(\O)}+\|D^2_vf\|_{\Ck^\a(\O)}+\|(\p_t+v\cdot\nabla_x)f\|_{\Ck^\a(\O)}. 
 \end{equation*}
\end{definition}

The consistence between these two definitions is given by \cite[Lemma 2.7]{ISkinetic} (see also \cite[Lemma 2.4]{IM}), a result that we state here as follows. 
\begin{lemma}\label{kinetic-degree}
Let $\a\in(0,1)$ and $f\in\Ck^{2+\a}(Q_2)$. Then, there exists some constant $C>0$ depending only on the dimension $d$ such that  
\begin{equation*}
\|\nabla_vf\|_{\Ck^{\a}(Q_1)} \le C \|f\|_{\Ck^{1+\a}(Q_2)},\quad
\|D^2_vf\|_{\Ck^\a(Q_1)}+\|(\p_t+v\cdot\nabla_x)f\|_{\Ck^\a(Q_1)} \le C \|f\|_{\Ck^{2+\a}(Q_2)}. 
\end{equation*}
\end{lemma}

\begin{remark}\label{holderremark}
	For $\a>2$, one can easily check that the polynomial $p_0$ in \eqref{holdercontinuousineq} has the form
	\begin{equation*}
	p_0(t,x,v)= f(z_0)+(\p_t+v_0\cdot\nabla_x)f(z_0)\,t+\nabla_vf(z_0)\cdot v + \frac{1}{2} D^2_vf(z_0)\,v\cdot v+\ldots. 
	\end{equation*}
	In particular, if $\a\in(2,3)$, the polynomial expansion is independent of the $x$-variable. 
\end{remark}

\begin{remark}
	A subtle difference between $\Ck^2$ and $\Ckin^2$ comes from the fact that for $f\in\Ck^2$, $D^2_vf$, $(\p_t+v\cdot\nabla_x)f$ are lying in $L^\infty$ rather than ${\mathcal C}^0$. 
\end{remark}

We will also employ the following notions of weighted H\"older norms in Section~\ref{linearsection}. 

\begin{definition}\label{weighted-norm}
Let $z=(t,x,v)\in\O:=(0,T]\times\R^d\times\R^d$ with $T\in\R_+$. 
For $f\in\Ck^{\a}(\O)$ with $\a>0$ and $\s\in\R$, we define 
\begin{equation*}
[f]_{0}^{(\sigma)}:=
\sup\nolimits_{z\in\O}\iota^{\sigma}[f]_{\Ck^0(Q_\iota(z))},\ \ 
[f]_{\a}^{(\sigma)}:=
\sup\nolimits_{z\in\O}\iota^{\a+\sigma}[f]_{\Ck^\a(Q_\iota(z))},\ \ 
\|f\|_\a^{(\sigma)}:= [f]_0^{(\sigma)}+[f]_\a^{(\sigma)},  
\end{equation*} 
where $\iota:=\min\big\{1,t^\frac{1}{2}\big\}$ measures the distance between $z$ and the (parabolic) boundary of $\O$.
\end{definition}

\subsection{Other notations}
Throughout the article, $B_R$ denotes the Euclidean ball in $\R^d$ centered at the origin with radius $R>0$. 
We employ the Japanese bracket defined as $\la\cdot\ra:=\left(1+|\cdot|^2\right)^\frac{1}{2}$. 
By abuse of notation, $\la\cdot\ra$ will also denote the velocity mean in Section~\ref{diffusionlimit}. 

Moreover, we assume $0<\lambda<\Lambda$. 
We denote by $C$ a \emph{universal} constant, that is to say a constant depending only on $\b,d,\lambda,\Lambda,\a,\s,\a_0$ specified in context. 
Finally, we write $X\lesssim Y$ to denote that $X\le CY$ for some universal constant $C>0$, 
and $X\lesssim_{q} Y$ to denote that $X\le C_qY$ for $C_q>0$ depending only on universal constants and the quantity $q$.

\section{Kolmogorov-Fokker-Planck equation}\label{linearsection}
This section is devoted to the study of the Cauchy problem associated to the operator~\eqref{k}, 
\begin{equation}\label{linearequation}
\left\{ 
\begin{aligned}
\ &\L_1 f :=\left(\partial_t+v\cdot\nabla_{x}\right)f-\trace(AD_v^{2}f)-B\cdot\nabla_{v}f=s \quad{\rm in\ } (0,T] \times\R^{d}\times\R^{d}, \\
&\; f(0,x,v)=\fin(x,v) \quad {\rm in\ } \R^{d}\times\R^{d}, \\
\end{aligned}
\right. 
\end{equation} 
where the $d\times d$ symmetric matrix $A(t,x,v)$ and the $d$-dimensional vector $B(t,x,v)$ satisfy the following condition, 
\begin{equation}\label{schaudercondition}
\left\{ 
\begin{aligned}
\ & A\,\xi\cdot\xi\ge\lambda|\xi|^2{\quad\rm\ for\ any\ }\xi\in\R^d,\\
& \|A\|_{\Ck^\a} +\|B\|_{\Ck^\a}\le\Lambda, \\
\end{aligned}
\right. 
\end{equation} 
where $\a\in(0,1)$ and the norm $\|\cdot\|_{\Ck^\a(\O)}$ of matrix denotes the summation of the norm of each entry. The boundedness condition at infinity means that the solution shall be bounded, which is intended for the validity of maximum principle; see the proof of Lemma~\ref{max} below.

The aim of this section is to solve the Cauchy problem~\ref{linearequation} by virtue of the weighted H\"older norm (Definition~\ref{weighted-norm}) and by means of the standard continuity method combined with Schauder type estimates. 
One may refer to \cite[Subsection 6.5]{GilbargTrudinger} for the corresponding treatment in classical elliptic theory.  

Throughout this section we work with the domain $\O:=(0,T] \times\R^{d} \times\R^{d}$, with $T\in\R_+$. We shed light on the fact that all of the results below can be restricted to $(0,T] \times \T^{d} \times\R^{d}$ whenever required.

\subsection{Schauder estimates}
In order to apply the continuity method, first of all one needs to prove a global a priori estimate for solutions to \eqref{linearequation} with respect to the weighted H\"older norm. 
In the kinetic setting, we have at our disposal the following interior Schauder estimates proved in \cite[Theorem 3.9]{IM}.

\begin{proposition}[Interior Schauder estimate]\label{interiorschauder} 
Let the constant $\a\in(0,1)$ be given and the cylinder $Q_{2r}(z_0)\subset\O$ with $r\in(0,1]$. 
If $f$ verifies the equation~\eqref{k} with the condition~\eqref{schaudercondition} in $Q_{2r}(z_0)$ and $s\in\Ck^\a(Q_{2r}(z_0))$, then we have 
\begin{equation}\label{schauderIM}
r^{2+\a} [f]_{\Ck^{2+\a}(Q_r(z_0))}
\lesssim \|f\|_{L^\infty(Q_{2r}(z_0))}+ r^{2+\a} [s]_{\Ck^{\a}(Q_{2r}(z_0))}; 
\end{equation}
in particular, the right hand side controls $r^2\|\left(\p_t+v\cdot\nabla_x\right)f\|_{L^\infty(Q_r(z_0))} +r^2\|D_v^2f\|_{L^\infty(Q_r(z_0))}$. 
\end{proposition}

First of all, we enhance this result to a global estimate for the Cauchy problem \eqref{linearequation} under a vanishing condition for the initial data.

\begin{proposition}[Global Schauder estimate]\label{globalschauder} 
Let $\O=(0,T] \times\R^{d} \times\R^{d}$, and the constants $\a\in(0,1)$, $\s\in(0,2)$ be universal, $s\in\Ck^\a(\O)$ such that $\|s\|_{\alpha}^{(2-\sigma)}<\infty$ and $f$ be a bounded solution to the Cauchy problem~\eqref{linearequation} under the condition \eqref{schaudercondition} in $\O$. 
If the initial data $\fin=0$, then we have 
\begin{equation*}
\|f\|_{2+\alpha}^{(-\sigma)}\lesssim\|s\|_{\alpha}^{(2-\sigma)}. 
\end{equation*}
\end{proposition}
	
\begin{proof}
In view of Proposition~\ref{interiorschauder}, it suffices to deal with the estimates around the initial time. Without loss of generality, we assume $T\le1$. 

Let $z_0=(t_0,x_0,v_0)\in\O$ and $2r=t_0^\frac{1}{2}$. 
Applying the interior Schauder estimate~\eqref{schauderIM} yields 
\begin{equation*}
r^{2+\alpha}[f]_{\Ck^{2+\a}(Q_{r}(z_0))}
\lesssim \|f\|_{L^\infty(Q_{2r}(z_0))}+ r^{2+\alpha}[s]_{\Ck^\alpha(Q_{2r}(z_0))}. 
\end{equation*}
It then follows from the arbitrariness of $z_0$ that, for any $\s<2$ such that $ [f]_{0}^{(-\s)}<\infty$, 
\begin{equation}\label{schauderIM0}
[f]_{2+\a}^{(-\s)} 
\lesssim [f]_{0}^{(-\s)}  + [s]_{\a}^{(2-\s)}. 
\end{equation}
With $\s\in(0,2)$, observing that 
\begin{align*}
\L_1\left(\frac{2}{\s}[s]_{0}^{(2-\s)}t^{\frac{\s}{2}}\pm f\right)=[s]_{0}^{(2-\s)}t^{\frac{\s}{2}-1}\pm s\ge0 
{\quad\rm in\ }\O\\
 \frac{2}{\s}[s]_{0}^{(2-\s)}t^\frac{\s}{2}\pm f=0 {\quad\rm on\ }\{t=0\}&, 
\end{align*} 
we apply the maximum principle (Lemma~\ref{max}) to the function $\frac{2}{\s}[s]_{0}^{(2-\s)}t^{\frac{\s}{2}}\pm f$ 
to deduce that 
$\pm t^{-\frac{\s}{2}} f\le \frac{2}{\s}[s]_{0}^{(2-\s)},$ 
which means 
\begin{equation*}
[f]_{0}^{(-\sigma)}\lesssim [s]_{0}^{(2-\sigma)}.
\end{equation*} 
Combining this estimate with \eqref{schauderIM0}, we get the desired result. 
\end{proof}

\subsection{Cauchy problem for the linear equation}
The goal of this subsection is to prove the well-posedness of the Cauchy problem~\eqref{linearequation} with H\"older continuous coefficients. 
\begin{proposition}\label{variable}
Let $\O=(0,T] \times\R^{d} \times\R^{d}$, and the constants $\alpha\in(0,1)$, $\s\in(0,2)$ be universal. Assume that 
\begin{equation}\label{schaudercondition'}
\left\{ 
\begin{aligned}
\ & A\,\xi\cdot\xi\ge\lambda|\xi|^2{\quad\rm\ for\ any\ }\xi\in\R^d,\\
& \|A\|_{\Ck^\a(\O)} +\|\la v\ra^{-1}B\|_{\Ck^\a(\O)}\le\Lambda.\\
\end{aligned}
\right. 
\end{equation}

\noindent
Then, for any $s\in\Ck^\a(\O)$ such that $\|s\|_{\alpha}^{(2-\sigma)}<\infty$ and $\fin\in {\mathcal C}^0(\R^d\times\R^d)$, 
there exists a unique bounded solution $f\in \Ck^{2+\a}(\O)$ to the Cauchy problem~\eqref{linearequation}. 
\end{proposition}

\begin{remark}
In contrast with \eqref{schaudercondition}, the condition~\eqref{schaudercondition'} is weaker, which allows the coefficients of $B$ to not necessarily be bounded globally. This fact will be applied to the Ornstein-Uhlenbeck operator $\lfp=\left(\nabla_v-v\right)\cdot\nabla_v$ in Subsection~\ref{wellsection}. 
\end{remark}

The simplest possible setting of \eqref{linearequation} under the condition~\eqref{schaudercondition'} is recovered by choosing $A=I$ and $B=0$, which turns out to be the classical Kolmogorov operator $\L_0:=\p_t+v\cdot\nabla_x-\Delta_v$. 
Such operator was first studied in \cite{K1}, where its fundamental solution was calculated explicitly, 
\begin{equation}\label{fun-sol}
\G(t,x,v) = 
\begin{cases}
\ \left( \frac{\sqrt{3}}{2 \pi t^{2}} \right)^{d} 
 e^{-\frac{3\left|x+\frac{t}{2}v\right|^{2}}{t^{3}}-\tfrac{|v|^{2}}{4t}}&{\rm for\ }t>0,\\
\ \ 0 &{\rm for\ }t\le 0.
\end{cases}
\end{equation}
One can easily see that $\G$ is smooth outside of its pole (the origin). 
In fact, in this latter case the following result holds true.

\begin{lemma}\label{constant}
Let $\O=(0,T] \times\R^{d} \times\R^{d}$ and $\alpha\in(0,1)$. For any $s\in\Ck^\a(\O)$ such that $\|s\|_{\alpha}^{(2-\sigma)}<\infty$, the function 
\begin{equation}\label{representation}
f(t,x,v)=\int_{\R^d\times\R^d} \G\left((\t,\x,\y)^{-1}\circ(t,x,v)\right)s(\t,\x,\y) \dif\t \dif\xi\dif\y
\end{equation}
is the unique bounded solution in $\Ck^{2+\a}(\O)$ to \eqref{linearequation} with $\L_1$ replaced by $\L_0$ and $\fin=0$. 
\end{lemma} 
\begin{remark}
When the spatial domain is $\T^d$, one can apply the Green function 
${G}(t,x,v):=\sum\nolimits_{n\in\Z^d}\G(t,x+n,v)$, which is well-defined due to the decay of $\G$. 
\end{remark}

We are now in a position to apply the standard continuity method to derive Proposition~\ref{variable}. 
\begin{proof}[Proof of Proposition~\ref{variable}]
We split the proof in three steps. 
In the first step, we establish the case for vanishing initial data under the stronger assumption~\eqref{schaudercondition}. 
We point out that the assumption on the coefficient $B$ can be weakened in the second step. 
Finally, we deal with general continuous initial data.

\medskip\noindent{\textbf{Step 1.}}	
Assume that $\fin=0$ and the condition~\eqref{schaudercondition} holds. 
Let the constant $\s\in(0,2)$ be fixed and consider the Banach space 
$\Y:=\left(\Ck^{2+\alpha}(\O),\|\cdot\|_{2+\alpha}^{(-\sigma)}\right)$. 
In particular, every function lying in $\Y$ vanishes at $t=0$. 

For $\t \in [0,1]$, we define the operator
$\L_\t:=(1-\t)\L_0+\t\L_1$, which can be written in the form of 
\begin{equation*}
\L_\t=\partial_t+v\cdot\nabla_{x}-\trace(A_\t D_v^{2}\cdot)-\t B\cdot\nabla_{v},
\end{equation*}
where its coefficients $A_\t:=(1-\t)I+\t A$ and $\t B$ still satisfy the condition~\eqref{schaudercondition} (with $\lambda$, $\Lambda$ replaced by $\min\{1,\lambda\}$, $\max\{1,\Lambda\}$, respectively). 
For any $w\in \Y$, we have
\begin{equation}\label{contraction0}
\|\L_\t w\|_{\alpha}^{(2-\sigma)}
\lesssim \left(1+\|A_\t\|_{\alpha}^{(2)}\right)\|w\|_{2+\alpha}^{(-\sigma)}
+\|B\|_{\alpha}^{(2)}\|w\|_{1+\alpha}^{(-\sigma)}
\lesssim\|w\|_{2+\alpha}^{(-\sigma)}.
\end{equation}

Let the set $\mathcal{I}$ be the collection of $\t\in[0,1]$ such that
the Cauchy problem~\eqref{zerocauchy} is solvable for any $s\in \Ck^\alpha(\O)$ with $\|s\|_{\alpha}^{(2-\sigma)}<\infty$: 
there is a unique bounded solution $f\in \Y$ verifying 
\begin{equation}\label{zerocauchy}
\left\{ 
\begin{aligned}
\ & \L_\t f=s  {\quad\rm in\ } \O, \\
\ &\;f(0,x,v)=0 {\quad\rm in\ } \R^{d}\times\R^{d}. 
\end{aligned}
\right. 
\end{equation}  
By Lemma~\ref{constant}, we see that $0\in\mathcal{I}$; in particular, $\mathcal{I}$ is not empty. 

It now suffices to show that $1\in\mathcal{I}$.
Pick $\t_0\in\mathcal{I}$, then the global Schauder estimate provided by Proposition~\ref{globalschauder} implies that, for any $s\in\Ck^\a(\O)$ with $\|s\|_{\alpha}^{(2-\sigma)}<\infty$, $f=\L^{-1}_{\t_0} s$ satisfies 
\begin{equation}\label{contraction}
\|\L^{-1}_{\t_0} s\|_{2+\alpha}^{(-\sigma)}
\lesssim\|s\|_{\alpha}^{(2-\sigma)}. 
\end{equation}
For any $w\in \Y$, since $\t_0\in\mathcal{I}$ and \eqref{contraction0} holds, 
the following Cauchy problem is solvable for any $s\in \Ck^\alpha(\O)$ with $\|s\|_{\alpha}^{(2-\sigma)}<\infty$, 
\begin{equation*}
\left\{ 
\begin{aligned}
\ & \L_{\t_0} f=s+(\t-\t_0)\left(\L_0-\L_1\right)w {\quad\rm in\ } \O,\\
\ &\; f(0,x,v)=0 {\quad\rm in\ } \R^{d}\times\R^{d}. 
\end{aligned}
\right. 
\end{equation*} 
Thus, we can define the mapping $F:\Y\rightarrow \Y$ by setting $F(w)=f$.
Armed with \eqref{contraction} and \eqref{contraction0}, there exists a universal constant $C>0$ such that, for any $u,w\in \Y$, 
\begin{equation*}
\|F(u)-F(w)\|_{2+\alpha}^{(-\sigma)} 
\le C|\t-\t_0| \left\|\left(\L_0-\L_1\right)(u-w)\right\|_{\a}^{(2-\sigma)}
\le C|\t-\t_0|\|u-w\|_{2+\alpha}^{(-\sigma)}. 
\end{equation*}
Hence, $F$ is a contraction mapping, provided that $|\t-\t_0|\le \delta:=\frac{1}{2C}$. 
Then, $F$ gives a unique fixed point $f\in \Y$, that is the unique bounded solution to the Cauchy problem \eqref{zerocauchy} in $\Y$. 
By dividing the interval $[0,1]$ into subintervals of length less than $\delta$, we conclude that $1\in\mathcal{I}$.

\medskip\noindent{\textbf{Step 2.}}
If $\fin=0$ and the condition~\eqref{schaudercondition'} holds, we approximate the coefficient $B$ by $B_n:=B\varrho_n$, 
where $\varrho_n(v):=\varrho_1\big(\frac{v}{n}\big)$ for $v\in\R^d$, $n\in\N_+$, and $\varrho_1\in{\mathcal C}_c^\infty(B_2)$ is valued in $[0,1]$ such that $\varrho_1\equiv1$ in $B_1$. 
Then, for each $n\in\N_+$, the result obtained in the previous step provides a bounded solution $f_n$ to the equation~\eqref{linearequation} with $B$ replaced by $B_n$. 
Indeed, applying the maximum principle (Lemma~\ref{max}) to the function $\pm f-e^t\sup_\O |s|$ implies that $\sup_\O |f_n|\le e^T\sup_\O |s|$. 
Thanks to the interior Schauder estimate (Proposition~\ref{interiorschauder}), for any compact subset $K\subset\O$, $\{f_n\}_{n\ge N}$ is precompact in $\Ckin^2(K)$, provided that $N$ (depending on $K$) is large enough. 
Sending $n\rightarrow\infty$ in the equation satisfied by $f_n$ yields that the limit function $f\in\Ck^{2+\a}(\O)$ is a bounded solution to \eqref{linearequation}, which satisfies $\sup_\O |f|\le e^T\sup_\O |s|$.

\medskip\noindent{\textbf{Step 3.}}		
For general $\fin \in {\mathcal C}^0(\R^d\times\R^d)$, we approximate $\fin$ uniformly as $\varepsilon\rightarrow0$ by a sequence of smooth functions $\{\fin^\varepsilon\}$ on $\R^d\times\R^d$. 
Thus, the function $f-\fin^\varepsilon$ is a solution to \eqref{linearequation}, with the source term equal to $s-\L_1\fin^\varepsilon$, and associated with the vanishing initial data. 
The procedure presented in the previous steps ensures a unique bounded solution $f^\varepsilon$ to \eqref{linearequation} for each $\fin^\varepsilon$. 

The uniform convergence of $\{\fin^\varepsilon\}$ and the maximum principle (Lemma~\ref{max}) implies the uniform convergence of $\{f^\varepsilon\}$. 
We may denote its limit by $f\in{\mathcal C}^0(\overline{\O})$, which satisfies $f(0,x,v)=\fin(x,v)$ on $\R^d\times\R^d$. 
The interior Schauder estimate again implies that $\{f^\varepsilon\}$ is precompact in $\Ckin^2(K)$ for any compact subset $K\subset\O$; 
then sending $\varepsilon\rightarrow0$ gives the solution $f\in\Ck^{2+\a}(\O)$ to the equation~\eqref{linearequation}. 
Its uniqueness is again given by the maximum principle. 
This concludes the proof. 
\end{proof}

\section{Well-posedness of the nonlinear model}\label{nonlinear}
This section is devoted to the proof of Theorem \ref{wellpose} (i), including a self-generating lower bound given in Subsection~\ref{lowersection},  the existence and uniqueness given in Subsection~\ref{wellsection} and smoothness a priori estimate given in Subsection~\ref{smoothsection}.

First of all, the Cauchy problem \eqref{nk} can be recasted in terms of the unknown function $g(t,x,v):= \m(v)^{-\frac{1}{2}}f(t,x,v)$ with $g_{\rm in}(x,v):=\m(v)^{-\frac{1}{2}}\fin(x,v)$ as follows, 
\begin{align}\label{gnk}
&\left\{ 
\begin{aligned}
\ &\left(\p_t+ v\cdot\nabla_x\right)g=\RO[g]\, \UO[g], \\
\ &\; g(0,x,v)= g_{\rm in}(x,v), \\
\end{aligned}
\right. 
\end{align} 
where $\RO[g]$ and $\UO[g]$ on the right hand side are defined by  
\begin{equation*}
\RO[g]:= \left(\int_{\R^{d}} g\m^{\frac{1}{2}} \dif v\right)^\b {\quad\rm and\quad}
\UO[g]:=\m^{-\frac{1}{2}}\nabla_v\left(\m\nabla_{v}\left(\m^{-\frac{1}{2}}g\right)\right)
=\Delta_vg+\left(\frac{d}{2}-\frac{|v|^{2}}{4}\right)g. 
\end{equation*}
The main advantage of this formulation is that it allows us to get rid of the first order term in $v$ and the zero order term is bounded, since $g$ is bounded from above by a Maxwellian. 

For convenience, we are also concerned with the substitution $h(t,x,v):= \m(v)^{-1}f(t,x,v)$ and the Ornstein-Uhlenbeck operator $\lfp:=\left(\nabla_v-v\right)\cdot\nabla_v$. The equation~\eqref{nk} is then equivalent to 
\begin{equation}\label{rnk}
\left(\p_t+v\cdot\nabla_x\right) h(t,x,v) = \rfp_h(t,x)\,\lfp h(t,x,v),\quad 
\rfp_h(t,x):= \left(\int_{\R^{d}} h(t,x,v)\dif\m\right)^\b. 
\end{equation} 
In contrast with \eqref{nk}, the zero order term disappears. 
Let us begin by exhibiting the global bounds of solutions to \eqref{rnk} in $(0,T)\times\T^d\times\R^d$, which is a variant of \cite[Lemma 4.1]{IM}.

\begin{lemma}[Global bounds]\label{Gaussian}
Let $a(t,x)\in L^\infty((0,T)\times\T^d)$ be nonnegative. 
Assume that $h(t,x,v)\in L^\infty \big((0,T);H^1\big(\T^d\times\R^d,\dif m\big)\big)$ verifies the equation $\left(\p_t+v\cdot\nabla_x\right)h=a\lfp h$ in $(0,T)\times\T^d\times\R^d$ in the sense of distributions. 
If $h(0,\cdot,\cdot)\le\Lambda$ in $\T^d\times\R^d$, then $h\le\Lambda$ in $(0,T)\times\T^d\times\R^d$; 
if $h(0,\cdot,\cdot)\ge\lambda$ in $\T^d\times\R^d$, then $h\ge\lambda$ in $(0,T)\times\T^d\times\R^d$.
\end{lemma} 
\begin{proof}
Integrating the equation $\left(\p_t+v\cdot\nabla_x\right)(h-\Lambda)=a\lfp(h-\Lambda)$ against $(h-\Lambda)_+$ yields
\begin{align*}
\frac{1}{2}\int_{\T^d\times\R^d}\left[(h(t,\cdot,\cdot)-\Lambda)_+^2-(h(0,\cdot,\cdot)-\Lambda)_+^2\right]\dif m&\\
=-\int_{[0,t]\times\T^d\times\R^d}a\left|\nabla_v(h-\Lambda)_+\right|^2 \dif t \dif m\le0&, 
\end{align*}
for any $t\in[0,T]$. 
This means that the upper bound is preserved along time. 
Similarly, the lower bound can be obtained by integrating the equation $\left(\p_t+v\cdot\nabla_x\right)(\lambda-h)=a\lfp(\lambda-h)$ against $(\lambda-h)_+$. 
\end{proof}

In particular, the above global bounds preserving result holds for solutions to the equations~\eqref{rnk} and \eqref{epnk} in $(0,T)\times\T^d\times\R^d$. 
We will also apply such result to the substitution $g=\m^\frac{1}{2}h$ appearing in the following Subsection~\ref{wellsection} below. 
Unless otherwise specified, throughout this section we set the domain  $\O:=(0,T]\times\T^d\times\R^d$ with $T\in\R_+$. 
Nevertheless, as it is specified in Remark~\ref{lowerwholeremark}, Corollary~\ref{existwhole} and Proposition~\ref{unique} below, 
the results of this section also hold if the spatial domain is $\R^d$.

\subsection{Self-generating lower bound}\label{lowersection}
Throughout this subsection, we assume that the bounded solution $h$ of \eqref{nk} lies below the universal constant $\Lambda$, which is guaranteed by Lemma~\ref{Gaussian} if the initial data lies below $\Lambda$. 
The aim of this subsection is to show the following positivity-spreading result. We remark that this proposition only relies on the mixing structure of the classical parabolic-type maximum principle and the transport operator, but not on the structure of the mass conservation. 

\begin{proposition}[Lower bound]\label{lower}
Let $\delta>0$, $\underline{T}\in(0,T)$ and $h$ be a bounded solution to \eqref{rnk} in $\O$ satisfying
\begin{equation}\label{lowers00}
h(0,x,v)\ge\delta\mathbbm{1}_{\left\{|x-x_0|<r,\,|v-v_0|<r\right\}},
\end{equation}
for some $(x_0,v_0)\in\T^d\times\R^d$. 
Then, there exist two positive continuous functions $\eta_1(t),\eta_2(t)$ on $(0,T]$ depending only on universal constants, $T,\delta,r$ and $v_0$ such that for any $(t,x,v)\in\O$, 
\begin{equation}\label{lowers0}
h(t,x,v)\ge \eta_1(t) e^{-\eta_2(t)|v|^2}. 
\end{equation}
\end{proposition}
\begin{remark}
In particular, the functions $\eta_1(t),\eta_2(t)$ are positive and bounded on any compact subset of $(0,T]$, but $\eta_1$ might degenerate to zero and $\eta_2$ may go to infinity as $t$ tends to zero or infinity.
\end{remark}

\begin{remark}\label{lowerwholeremark}
If one is concerned with the problem in the whole space, that is $\O=(0,T]\times\R^d\times\R^d$, we can proceed along the same lines of the proof in Appendix~\ref{lowerapp} to see that \eqref{lowers00} implies the following lower bound 
\begin{equation}\label{lowerwhole}
h(t,x,v)\ge \eta_1(t,x)^{-1} e^{-\eta_2(t,x)|v|^4}, 
\end{equation}
where the functions $\eta_1(t,x),\eta_2(t,x)$ on $(0,T]\times\R^d$ are positive, continuous and only depend on universal constants, $T,\delta,r$ and $v_0$. 
Compared with \eqref{lowers0}, $\eta_1(t,x),\eta_2(t,x)$ lose the uniformity in $x$ as $\R^d$ is not compact (see Step~3 of the proof of the proposition in Appendix~\ref{lowerapp}). 
Besides, the exponential tail with respect to $v$ cannot be improved to a Gaussian type, since there is no uniform-in-$x$ lower bound on the local mass $\int h\dif\m$ so that Step~4 in Appendix~\ref{lowerapp} fails. 
\end{remark}

Note that the proof of the proposition is composed mainly of two lemmas. 
On the one hand, Lemma~\ref{lowerlemma1} below extends the lower bounds forward a short time from a neighborhood of any given point in $\T^d\times\R^d$ and at any given time. 
On the other hand, Lemma~\ref{lowerlemma2} below is used to spread the lower bound for all velocities. 
The spreading of the lower bound in space is given by selecting proper velocity to transport the positivity which is guaranteed by Lemma \ref{lowerlemma1}. 
By applying these two lemmas repeatedly, as proposed in \cite{HST3}, we are able to control the solution from below for any finite time. 
We postpone the full proof of Proposition~\ref{lower}, 
obtained by combining the previously introduced lemmas, 
until Appendix~\ref{lowerapp}.

\begin{lemma}[Lower bound forward in time]\label{lowerlemma1}
Let $\delta,\tau,r\in(0,1]$ and $h$ be a bounded solution to \eqref{rnk} in $\O$ with 
\begin{equation*}
h(0,x,v)\ge\delta\mathbbm{1}_{\left\{|x-x_0|<r,\, |v-v_0|<\frac{r}{\tau}\right\}},
\end{equation*}
for some $(x_0,v_0)\in\T^d\times\R^d$. Then, there exists some universal constant $c_0>0$ such that 
\begin{equation*}
h\ge\frac{\delta}{8}\mathbbm{1}_\mathcal{P},{\quad}
\mathcal{P}:=\left\{t\le \min\big\{T,\tau,c_0\left\la\tau r^{-1}\right\ra^{-2}\!\la v_0\ra^{-2}\big\},\, |x-x_0-tv|<\frac{r}{2},\, |v-v_0|<\frac{r}{2\tau}\right\}.
\end{equation*}
\end{lemma}
	
\begin{proof}
Let us consider the barrier function $$\underline{h}(t,x,v):=-C_0t+\frac{\delta}{2}\left(1-r^{-2}|x-x_0-tv|^2-\tau^2r^{-2}|v-v_0|^2\right),$$ 
where the constant $C_0>0$ is to be determined. 
The region $\mathcal{Q}:=\big\{t\le\min\{T,\tau\},\,|x-x_0-tv|^2+\tau^2|v-v_0|^2<r^2\big\}$ contains $\mathcal{P}$. 
A direct computation yields that
\begin{equation*}
|\lfp\underline{h}|\le |\Delta_v\underline{h}|+ |v\cdot\nabla_v\underline{h}|
\lesssim \delta\left\la\tau r^{-1}\right\ra^{2}\!\la v_0\ra^2 {\quad\rm in\ }\mathcal{Q}. 
\end{equation*}
By choosing $C_0:=\frac{1}{8c_0}\delta\left\la\tau r^{-1}\right\ra^{2}\!\la v_0\ra^2$ for some (small) universal constant $c_0>0$, we have
\begin{equation}\label{hunder}
\left(\partial_t+v\cdot\nabla_x-\rfp_h\,\lfp\right)\underline{h}\le -C_0+\Lambda^\b|\lfp\underline{h}|<0 
{\quad\rm in\ }\mathcal{Q}.  
\end{equation}
Besides, $\underline{h}(t,x,v)\ge\frac{\delta}{8}$ in $\left\{t\le c_0\left\la\tau r^{-1}\right\ra^{-2}\!\la v_0\ra^{-2},\; |x-x_0-tv|^2+\tau^2|v-v_0|^2<\frac{r^2}{2}\right\}\supset\mathcal{P}$. 

Then, observing that $\underline{h}-h\le0$ on the parabolic boundary $\{t=0\}\cap\mathcal{Q}$,  $\big\{t\le\min\{T,\tau\},\,|x-x_0-tv|^2+\tau^2|v-v_0|^2=r^2\big\}$, and applying the classical maximum principle to $\underline{h}-h$ in $\mathcal{Q}$ yields the result. 
\end{proof}

The spreading of lower bound to all velocities relies on the construction of a Harnack chain through iterative application of the local Harnack inequality \cite[Theorem 1.6]{GIMV}. 
Although some coefficients of the equation~\eqref{rnk} are unbounded globally over $v\in\R^d$, we remark that their local boundedness is sufficient for us to achieve the result through a careful study on the rescaling during the construction of the Harnack chain. 

\begin{lemma}[Lower bound for all velocities]\label{lowerlemma2}
Let $\delta>0$, $T,R\in(0,1]$, $\underline{T}\in(0,T)$ and $h$ be a bounded solution to \eqref{rnk} in $\O$ such that, for any $t\in[0,T]$,  
\begin{equation}\label{lower2assumption}
h(t,x,v)\ge\delta\mathbbm{1}_{\left\{|x-x_0-tv_0|<R,\,|v-v_0|<R\right\}},
\end{equation}
for some $(x_0,v_0)\in\T^d\times\R^d$. Then, there exists some (large) constant $\underline{C}>0$ depending only on universal constants, $\underline{T},\delta,R$ and $v_0$ such that, for any $t\in[\underline{T},T]$, we have 
\begin{equation}\label{lower2bound}
h(t,x,v)\ge \underline{C}^{-1} e^{-\underline{C}|v|^4}\mathbbm{1}_{\left\{|x-x_0-tv_0|<\frac{R}{2}\right\}}. 
\end{equation}
\end{lemma}
	
\begin{proof}
For any $z:=(t,x,v)\in\big\{t\in[\underline{T},T],\,|x-x_0-tv_0|<\frac{R}{2},\,v\in\R^d\big\}$, we will construct a finite sequence of points to reach $z$ from the region $\big\{t\le{T},\,|x-x_0-tv_0|<R,\,|v-v_0|<R\big\}$ where the solution is positive by the assumption. 
In particular, $x$ does not exit this region. 
The nonlocality of the coefficient $\rfp_h$, with the assumption \eqref{lower2assumption}, implies the nondegeneracy of the diffusion in velocity so that the positivity of the solution $h$ propagates over $v\in\R^d$ in a localized space region. 

\medskip\noindent{\textbf{Step 1.}}	Iterate the Harnack inequality. \\
For $i\in\{1,2,...,N+1\}$ with $N\in\N$, we define $z_{N+1}:=z$ and $z_{i}:=(t_i,x_i,v_i)$ by the relation 
$$z_i=z_{i+1}\circ S_r\left(-\t_1,0,-\t_2\frac{v-v_0}{|v-v_0|}\right),$$  
where the parameters $N,r,\tau_1,\tau_2>0$ are to be determined next. 	
Consider the function for $\tilde{z}:=(\tilde{t},\tilde{x},\tilde{v})\in Q_1$, 
\begin{equation*}
h_i(\tilde{z}):=h\left(z_i\circ S_r(\tilde{z})\right)
=h(t_i+r^2\tilde{t},\,x_i+r^3\tilde{x}+r^2\tilde{t}v_i,\,v_i+r\tilde{v}).
\end{equation*} 
We observe that if the following is true for any $\tilde{z}\in Q_1$ 
\begin{align}
&\;t_{i+1}+r^2\tilde{t}\in[0,T],\quad Nr\tau_2\le |v-v_0|, \label{htxv21}\\
&\left|x_{i+1}+r^3\tilde{x}+r^2\tilde{t}v_{i+1}-x_0-\left(t_{i+1}+r^2\tilde{t}\right)v_0\right|<R,\label{htxv22}
\end{align}
then the function $h_{i+1}(\tilde{z})$, for $1\le i\le N$, verifies the equation 
\begin{equation*}
\left(\partial_{\tilde{t}}+\tilde{v}\cdot\nabla_{\tilde{x}}\right)h_i
=\rfp_h\left(z_i\circ S_r(\tilde{z})\right)\, \left(\Delta_{\tilde{v}}h_i-r(v_i+r\tilde{v})\cdot\nabla_{\tilde{v}}h_i\right) {\rm\quad in\ } Q_1, 
\end{equation*}
where the coefficients satisfy 
		\begin{equation*}
		\delta^\b R^{d\b}\lesssim\rfp_h\lesssim 1  {\quad\rm and\quad}
		|r(v_i+r\tilde{v})|\le r(1+|v_0|+|v-v_0|)\le 1, 
		\end{equation*}
		provided that $r\le(1+|v_0|+|v-v_0|)^{-1}$. 
Applying the Harnack inequality \cite[Theorem 1.6]{GIMV} to $h_i$ implies that there exist constants $c_0,\tau_1\in(0,1)$ depending only on universal constants, $\delta$ and $R$ such that, for any $\tau_2\in[0,1-\t_1]$ and $1\le i\le N$, we have
\begin{equation}\label{hharnack}
h(z_{i+1})=h_{i+1}(0,0,0)\ge c_0h_{i+1}\Big(-\tau_1,\,0,\,-\tau_2\frac{v-v_0}{|v-v_0|}\Big)=c_0h(z_i).  
\end{equation}
Hence, it remains to determine the chain $\{z_i\}_{1\le i\le N+1}$ such that the conditions \eqref{htxv21}, \eqref{htxv22} hold and the point $z_1$ stays in the region $\big\{(t,x,v):t\le{T},\,|x-x_0-tv_0|<R,\,|v-v_0|<R\big\}$. 

\medskip\noindent{\textbf{Step 2.}}	Determine the Harnack chain (including $N,r,\tau_2$) from a proper starting time $t_1$.\\ 
For $M>0$, we set 
\begin{equation*}
t_1:=\max\left\{\frac{\underline{T}}{2},\,t-\frac{R}{8}(1+|v_0|+|v-v_0|)^{-1}\right\}
{\quad\rm and\quad}
r:=\frac{R}{M}(1+|v_0|+|v-v_0|)^{-2}.
\end{equation*}
Recalling that $T,R\in(0,1]$, by choosing 
$M\ge \frac{2}{\underline{T}} +\frac{\t_1}{1-\t_1}\left(8+\frac{2}{\underline{T}}\right)$, we have 
\begin{equation*}
r^2\le\frac{\underline{T}}{2} {\quad \rm and\quad}
\tau_2:=\frac{r\tau_1|v-v_0|}{t-t_1}\le 1-\t_1. 
\end{equation*}
To determine the parameter $M>0$, 
we point out that there exists some constant $\overline{C}$ depending only on universal constants, $\underline{T},\delta,R$ and $v_0$ such that $M\le\overline{C}$ and 
\begin{equation*}
N:=\frac{t-t_1}{r^2\tau_1}\in\N^+. 
\end{equation*}
Thus, $Nr\tau_2=|v-v_0|$. This setting then guarantees the condition \eqref{htxv21}. 

It also follows from the iteration relation that $v_1=v_0$, and for $1\le i\le N+1$,  
\begin{equation}\label{iterationrelation}
t_i=t_1+(i-1)r^2\t_1,\quad 
v_{i}=v_0+(i-1)r\t_2\frac{v-v_0}{|v-v_0|},\quad 
x_{i}=x-r^2\tau_1\sum\nolimits^{N}_{j=i}v_{j+1}. 
\end{equation}

\medskip\noindent{\textbf{Step 3.}} Determine the starting point $x_1$.\\	
For any $1\le i\le N$, we estimate the departure distance from the expression \eqref{iterationrelation}, 
\begin{equation*}
\left|x_{i+1}-x_1-(t_{i+1}-t_1)v_0\right| =\frac{i(i+1)}{2} r^3\tau_1\tau_2
\le N^2 r^3\tau_1\tau_2 =(t-t_1)|v-v_0|\le\frac{R}{8}. 
\end{equation*}
Therefore, for any $x\in B_\frac{R}{2}(x_0+tv_0)$, there exists some $x_1\in B_{\frac{5R}{8}}(x_0+t_1v_0)$ such that $x_{N+1}=x$. 
In this setting, for any $1\le i\le N$, we also have 
\begin{align*}
&\left|x_{i+1}+r^3\tilde{x}+r^2\tilde{t}v_{i+1}-x_0-\left(t_{i+1}+r^2\tilde{t}\right)v_0\right|\\
&\le \left|x_{i+1}-x_1-(t_{i+1}-t_1)v_0\right|+|x_1-x_0-t_1v_0| +r^2|r\tilde{x}+\tilde{t}v_{i+1}-\tilde{t}v_0|\\
&\le \frac{R}{8}+\frac{5R}{8}+r^2(1+|v-v_0|)<\frac{3R}{4}+\frac{R^2}{M^2}<R. 
\end{align*}
Thus, the condition~\eqref{htxv22} ensuring the inequality~\eqref{hharnack} is satisfied for $1\le i\le N$, which yields 
\begin{equation*}
h({t},{x},{v})\ge c_0^Nh(t_1,x_1,v_0) \ge \d e^{-N\log \frac{1}{c_0}}. 
\end{equation*}
Recalling that $c_0\in(0,1)$ appears in \eqref{hharnack} and  $N\le\frac{T\overline{C}^2\left(1+|v_0|+|v-v_0|\right)^4}{\t_1R^2}$, we obtain the desired result \eqref{lower2bound}. 
\end{proof}

\subsection{Existence and uniqueness}\label{wellsection}
Let us begin by summarizing some basic a priori estimates for solutions to the  equation~\eqref{gnk}. 
\begin{lemma}[H\"older estimates]\label{boundschauder} 
Let $\O_x=\T^d$ or $\R^d$, and $g$ be a solution to \eqref{gnk} in $[0,T]\times\O_x\times\R^d$ satisfying 
\begin{equation*}
\RO[g]\ge\lambda  {\quad\rm in\ } [0,T]\times\O_x {\quad\rm and\quad}
0\le g_{\rm in}\le\Lambda\m^{\frac{1}{2}} {\quad\rm in\ }\O_x\times\R^d. 
\end{equation*}
\begin{item}
\item [ {\bf\ \,(i)}] 
Let $\underline{T}\in(0,T)$ and $\d\in\big(0,\frac{1}{2}\big)$. 
There exists some universal constant $\a\in(0,1)$ such that, 
for any $Q_{2r}(z_{0})\subset[\underline{T},T]\times\O_x\times\R^d$, we have 
\begin{equation}\label{schauder1}
\| g \|_{\Ck^{2+\a}(Q_{r}(z_{0}))} \lesssim_{\underline{T},\d} \m^{\d} (v_{0}).
\end{equation}

\item [ {\bf\;(ii)}] 
If $g_{\rm in}\in {\mathcal C}^{\a_0}(\O_x\times\R^d)$ with (universal) $\a_0\in(0,1)$, 
then for any $\d\in\big(0,\frac{1}{2}\big)$, there exists some universal constant $\alpha\in(0,1)$ such that 
\begin{equation*}
\|g\|_{\Ck^\a\left([0,T]\times\O_x\times B_1(v_0)\right)}
\lesssim_\d \big(1+[g_{\rm in}]_{{\mathcal C}^{\a_0}(\O_x\times\R^d)}\big) \m^\d(v_0).
\end{equation*}
\end{item}
\end{lemma}

We remark that, armed with Lemma~\ref{Gaussian}, the assertions (i) and (ii) in the above lemma directly follow from \cite[Proposition 4.4]{IM} and \cite[Corollary 4.6]{YZ}, respectively.

\begin{proposition}[Existence]\label{existence}
For any $g_{\rm in}\in{\mathcal C}^0(\T^d\times\R^d)$ such that $0\le g_{\rm in}\le\Lambda\m^\frac{1}{2}$ in $\T^d\times\R^d$, there exists a (classical) solution $g$ to the Cauchy problem~\eqref{gnk} satisfying $0\le g\le\Lambda\m^\frac{1}{2}$ in $\O$. 
\end{proposition}

\begin{remark}\label{remark-weak}
For any given nonnegative continuous function $g_{\rm in}$ that is not identically zero, there is some point $(x_0,v_0)\in\T^d\times\R^d$ and some constants $\delta,r>0$ such that 
\begin{equation}\label{lowerdd}
g_{\rm in}(x,v)\ge\delta\mathbbm{1}_{\{|x-x_0|<r,|v-v_0|<r\}} 
{\quad\rm in\ }\T^d\times\R^d. 
\end{equation} 
We will see that the upper bound $g_{\rm in}\le\Lambda\m^\frac{1}{2}$ and the lower bound \eqref{lowerdd} assumptions on the initial data $g_{\rm in}$ (could be discontinuous) are sufficient to ensure the existence of solution $g\in\Ckin^2(\O)$ in the weak sense that for any $\phi\in {\mathcal C}^\infty_c([0,T)\times\T^d\times\R^d)$, 
\begin{equation}\label{existence0}
	\int_{\T^d\times\R^d} g_{\rm in}\,\phi\big|_{t=0}
	=\int_{\O}\left\{-g\left(\p_t+v\cdot\nabla_x\right)\phi
	+\RO[g]\nabla_vg\cdot\nabla_v\phi-\RO[g]\left(\frac{d}{2}-\frac{|v|^2}{4}\right)g\phi\right\}.
\end{equation}
As solutions become regular instantaneously, the difference between the weak solution and the classical one lies only in the continuity around the initial time. 
\end{remark}

\begin{proof}
Let us assume that $g_{\rm in}$ satisfies \eqref{lowerdd}, for some point $(x_0,v_0)\in\T^d\times\R^d$ and some constants $\delta,r>0$. 
By Proposition~\ref{lower}, for any solution $g$ to \eqref{gnk} and for any $\underline{T}\in(0,T)$, there is some $\lambda_*>0$ depending only on universal constants, $\underline{T},T,\delta,r$ and $v_0$ such that, 
\begin{equation}\label{exist-lower}
\RO[g](t,x)\ge \lambda_*
{\quad\rm in\ }[\underline{T},T]\times\T^d. 
\end{equation}	

\medskip\noindent{\textbf{Step 1.}}	
We first approximate the initial data $g_{\rm in}$ by 
$g_{\rm in}^\e:=g_{\rm in}*\varrho_\varepsilon+\varepsilon\m^{\frac{1}{2}}$, where
$\varrho_\varepsilon(x,v):=\frac{1}{\varepsilon^{2d}}\varrho_1\left(\frac{x}{\varepsilon},\frac{v}{\varepsilon}\right)$ 
with $(x,v)\in\T^d\times\R^d$, $\varepsilon\in(0,1]$, and $\varrho_1\in {\mathcal C}^\infty_c\left(B_1\times B_1\right)$ is a nonnegative bump function such that $\int_{\mathbb{R}^{2d}}\varrho_1=1$. 
Then, we have $\varepsilon\mu^{\frac{1}{2}}\le g_{\rm in}^\varepsilon\le(1+\Lambda)\mu^{\frac{1}{2}}$ in $\T^d\times\R^d$.

Let us fix $\varepsilon\in(0,1]$. 
In order to establish the existence of solution to \eqref{gnk} associated with the initial data $g_{\rm in}^\varepsilon$, 
we are going to find a fixed point of the mapping $F:w\mapsto{g}$ defined by solving the Cauchy problem, 
\begin{align}\label{nonexun}
&\left\{ 
\begin{aligned}
\ &\left(\p_t+v\cdot\nabla_x\right)g
= \RO[w]\,\UO[g]
{\ \ \rm in\ } \O, \\
\ &\; {g}(0,\cdot,\cdot)=g_{\rm in}^\varepsilon {\ \ \rm in\ }\T^d\times\R^d,\\
\end{aligned}
\right. 
\end{align} 
on the closed convex subset $\K$ of the Banach space $\Ck^\g(\overline{\O})$, 
\begin{equation*}
\K:=\Big\{
w\in\Ck^{\g}(\overline{\O})
:\,\|w\|_{\Ck^{\g}(\O)}\le\mathcal{N},\ \;
\varepsilon\mu^{\frac{1}{2}}\le w\le(1+\Lambda)\mu^{\frac{1}{2}}{\ \ \rm in\ }\overline{\O}
\Big\}, 
\end{equation*}	
where the constants ${\g}\in(0,1)$ and $\mathcal{N}>0$ are to be determined. 
We remark that the equation~\eqref{nonexun} is equivalent to 
\begin{equation*}
\left(\p_t+v\cdot\nabla_x\right)\big(\m^{-\frac{1}{2}}g\big)
= \RO[w]\,\lfp\big(\m^{-\frac{1}{2}}g\big). 
\end{equation*}
By Lemma~\ref{Gaussian} and the fact that $\RO\left[w\right]\ge \varepsilon$, we have  $\varepsilon\mu^{\frac{1}{2}}\le{g}\le(1+\Lambda)\mu^{\frac{1}{2}}$ in $\overline{\O}$. 
In particular, the lower order term $\big|\RO\left[w\right]\big(\frac{d}{2}-\frac{|v|^{2}}{4}\big)g\big|\lesssim1$ for any $w\in\K$. 
Thus, the global H\"older estimate \cite[Corollary 4.6]{YZ} implies that
there exist some constants $\g\in(0,1)$ and $\mathcal{N}>0$ depending only on universal constants and $\varepsilon$ such that $\|{g}\|_{\Ck^{2\g}(\O)}\le\mathcal{N}$, which also implies that the lower order term $\big|\RO\left[w\right]\big(\frac{d}{2}-\frac{|v|^{2}}{4}\big)g\big|$ is bounded in $\Ck^{2\g}(\O)$. 
It then follows from Proposition~\ref{variable} with the interior Schauder estimate (Proposition~\ref{interiorschauder}) that the mapping $F:\K\rightarrow\K\cap\Ck^{2\g}(\overline{\O})\cap\Ck^{2+2\g}(\O)$ is well-defined. 
Besides, with the help of the Arzelà–Ascoli theorem, we know that $F(\K)$ is precompact in $\Ck^\g(\overline{\O})$. 

As far as the continuity of $F$ is  concerned, we take a sequence $\{w_n\}$ converging to $w_\infty$ in $\Ck^\g(\overline{\O})$. 
Since $\{F(w_n)\}$ is precompact in $\Ck^\g(\overline{\O})$, there exists a converging subsequence whose limit is ${g}_\infty\in\Ck^\g(\overline{\O})$ which satisfies 
${g}_\infty(0,\cdot,\cdot)=g_{\rm in}^\varepsilon$ in $\T^d\times\R^d$. 
In view of the interior Schauder estimate (Proposition~\ref{interiorschauder}), $\{F(w_n)\}$ is precompact in $\Ckin^2(K)$ for any compact subset $K\subset\O$ and ${g}_\infty\in\Ckin^2(\O)\cap\mathcal{C}^0(\overline{\O})$. 
Sending $n\rightarrow\infty$ in \eqref{nonexun} satisfied by $(w,{g})=(w_n,F(w_n))$, we see that the equation \eqref{nonexun} also holds for the couple of limits $(w,{g})=(w_\infty,{g}_\infty)$. 
Then, applying the maximum principle (Lemma~\ref{max}) to 
\begin{align*}
&\left\{ 
\begin{aligned}
\ &\left(\p_t+v\cdot\nabla_x\right)\left(\m^{-\frac{1}{2}}(g_\infty-F(w_\infty))\right)
=\RO[w_\infty]\,\lfp \left(\m^{-\frac{1}{2}}\left(g_\infty-F(w_\infty)\right)\right)
{\quad\rm in\ }\O, \\
\ &\; (g_\infty-F(w_\infty))(0,\cdot,\cdot)=0 {\ \ \rm in\ }\T^d\times\R^d,\\
\end{aligned}
\right. 
\end{align*} 
we arrive at ${g}_\infty=F(w_\infty)$. 

Then for every $\varepsilon\in(0,1]$, we are allowed to apply the Schauder fixed point theorem (see for instance \cite[Corollary 11.2]{GilbargTrudinger}) to get $g^\varepsilon\in\Ckin^2(\O)\cap\mathcal{C}^0(\overline{\O})$ such that $F(g^\varepsilon)=g^\varepsilon$, which is a (classical) solution to \eqref{gnk} associated with the initial data $g_{\rm in}^\varepsilon$. 
		
\medskip
\noindent{\textbf{Step 2.}}	Passage to the limit.\\
Recalling the lower bound~\eqref{exist-lower} on the coefficient and the higher order H\"older estimate given by Lemma~\ref{boundschauder}~(i), we point out that for any $\underline{T}\in(0,T)$, $\{g^\varepsilon\}$ is uniformly bounded in $\Ck^{2+\a_*}([\underline{T},T]\times\T^d\times\R^d)$, for some constant $\a_*\in(0,1)$ with the same dependence as $\lambda_*$. 
Hence, $g^\varepsilon$ converges uniformly to $g$ in $\Ckin^2([\underline{T},T]\times\T^d\times\R^d)$, up to a subsequence. 
		
Write the equation satisfied by $g^\varepsilon$ in the weak formulation, that is, for any $\phi\in {\mathcal C}^\infty_c(\overline{\O})$, 
\begin{equation}\label{epexistence}
\begin{split}
&\int_{\T^d\times\R^d}\left[g^\varepsilon(T,x,v)\phi(T,x,v)-g^\varepsilon_{\rm in}(x,v)\phi(0,x,v)\right]\\
&=\int_{\O}\left\{g^\varepsilon\left(\p_t+v\cdot\nabla_x\right)\phi-\RO[g^\varepsilon]\nabla_vg^\varepsilon\cdot\nabla_v\phi+\RO[g^\varepsilon]\left(\frac{d}{2}-\frac{|v|^2}{4}\right)g^\varepsilon\phi\right\}.
\end{split}
\end{equation}
Combining the energy estimate derived by choosing $\phi=g^\varepsilon$ above with the upper bound of $g^\varepsilon$ provided by Lemma~\ref{Gaussian}, we have  
\begin{equation*}
\int_{\O}\left|\RO\left[g^\varepsilon\right]\nabla_vg^\varepsilon\right|^2
\lesssim \int_{\O}\RO\left[g^\varepsilon\right]|\nabla_vg^\varepsilon|^2
\le \frac{1}{2}\int_{\T^d\times\R^d}|g_{\rm in}^\varepsilon|^2
 +\int_{\O}\RO[g^\varepsilon]\left(\frac{d}{2}-\frac{|v|^2}{4}\right)|g^\varepsilon|^2
\lesssim 1.
\end{equation*}
Therefore, after passing to a subsequence, $\RO[g^\varepsilon]\nabla_vg^\varepsilon$ converges weakly in $L^2(\O)$. On account of its local uniform convergence, we know that its weak limit is $\RO[g]\nabla_vg$. 
Besides, since $\m^{-\frac{1}{2}}g^\varepsilon$ is uniformly bounded, by their local uniform convergence, we can also derive that the sequences $g^\varepsilon$ and $\RO[g^\varepsilon]\left(\frac{d}{2}-\frac{|v|^2}{4}\right) g^\varepsilon$ converge to  $g$ and $\RO[g]\left(\frac{d}{2}-\frac{|v|^2}{4}\right) g$, respectively, weakly in $L^2(\O)$, up to a subsequence. 		
Then, for any $\phi\in {\mathcal C}^\infty_c([0,T)\times\T^d\times\R^d)$, sending $\varepsilon\rightarrow0$ in \eqref{epexistence} gives \eqref{existence0}. 

Furthermore, if the initial data $g_{\rm in}$ is continuous, then the barrier function method shows that the continuity around the initial time depends only on the upper bound of the solution and the continuity of $g_{\rm in}$; see the derivation of the estimate~\eqref{coninitialbarrier} of a general type in Subsection~\ref{finitetimesection} below. 
Indeed, by \eqref{coninitialbarrier} (with $\e=1$, $R=|v_1|$, $h_\e=\m^{-\frac{1}{2}}g$ and $\hine=\m^{-\frac{1}{2}}g_{\rm in}$), we see that 
for any fixed $\d\in(0,1)$, $(x_1,v_1)\in\T^d\times\R^d$ and for any $(t,x,v)\in \big[0,\frac{\d}{4(1+|v_1|)}\big]\times B_\d(x_1,v_1)$, 
\begin{equation}\label{barrierinitial}
\begin{split}
|g(t,x,v)-g_{\rm in}(x_1,v_1)|
\lesssim&\; \d^{-2}\la v_1\ra^2\m^\frac{1}{2}(v_1)t + \d^{-2}\m^\frac{1}{2}(v_1)\left(|x-x_1-tv|^2+|v-v_1|^2\right)\\
 &+\m^\frac{1}{2}(v_1)\sup\nolimits_{B_\d(x_1,v_1)}|g_{\rm in}(x,v)-g_{\rm in}(x_1,v_1)|\\
\lesssim&\; \d^{-2}\left(t+|x-x_1|^2+\m^\frac{1}{2}(v_1)|v-v_1|^2\right)\\
  &+\sup\nolimits_{B_\d(x_1,v_1)}|g_{\rm in}(x,v)-g_{\rm in}(x_1,v_1)|. 
\end{split}
\end{equation}
It implies the continuity of the solution $g$ around $t=0$. 
This finishes the proof. 
\end{proof}

One may extend the above existence result to the case where the spacial domain $\T^d$ is replaced by $\R^d$. 
\begin{corollary}\label{existwhole}
For any $g_{\rm in}\in{\mathcal C}^0(\R^d\times\R^d)$ such that $0\le g_{\rm in}\le\Lambda\m^\frac{1}{2}$ in $\R^d\times\R^d$, there exists a solution $g$ to the Cauchy problem~\eqref{gnk} satisfying $0\le g\le\Lambda\m^\frac{1}{2}$ in $(0,T]\times\R^d\times\R^d$. 
If additionally $\hin=\m^{-\frac{1}{2}}g_{\rm in}$ satisfies $\hin\ge\lambda$ and $\hin-M_1\in L^1(\R^d\times\R^d,\dif m)$ for some constant $M_1>0$, then $h=\m^{-\frac{1}{2}}g$ satisfies $h\ge\lambda$  in $(0,T]\times\R^d\times\R^d$ and 
\begin{equation}\label{l2l1}
\|h-M_1\|_{L_t^\infty([0,T]; L^1(\R^d\times\R^d,\,\dif m))}\le \|\hin-M_1\|_{L^1(\R^d\times\R^d,\,\dif m)}.
\end{equation}
\end{corollary}

\begin{proof}
For $R>1$, we set $g_{\rm in}^R=g_{\rm in}\mathbbm{1}_{[-R+R^{-1},R-R^{-1}]^d}$ for $x\in[-R,R]^d$ with periodic extension to $\R^d$.
In the light of Proposition~\ref{existence}, we take a solution $g^R$ to the equation~\eqref{gnk} associated with the initial data $g_{\rm in}^R$ in $(0,T]\times[-R,R]^d\times\R^d$, where $[-R,R]^d$ is considered as a periodic box. 
After extracting a subsequence, we define the function $g:=\lim_{R\rightarrow\infty}g^R$ in $(0,T]\times\R^d\times\R^d$ pointwisely; 
furthermore, since $0\le \m^{-\frac{1}{2}}g^R\le\Lambda$ in $(0,T]\times[-R,R]^d\times\R^d$, 
we know that the limiting function satisfies $0\le \m^{-\frac{1}{2}}g\le\Lambda$ in $(0,T]\times\R^d\times\R^d$. 
Similarly, $\m^{-\frac{1}{2}}g_{\rm in}\ge\lambda$ in $\R^d\times\R^d$ implies that $\m^{-\frac{1}{2}}g\ge\lambda$ in $(0,T]\times\R^d\times\R^d$. 

Since the initial data is continuous, unless it is identically zero, we may assume that $g_{\rm in}\ge\delta\mathbbm{1}_{\{|x-x_0|<r,|v-v_0|<r\}}$, for some point $(x_0,v_0)\in\R^d\times\R^d$ and some constants $\delta,r>0$. 
Consider $R>|x_0|+r$. 
Applying the the lower bound of the solution given by \eqref{lowerwhole} yields that, for any compact subset $K\subset(0,T]\times\R^d\times\R^d$, the coefficient $\RO[g^R]\ge\lambda_*$, where the constant $\lambda_*>0$ only depends on universal constants,  $\delta,r,v_0$ and $K$.  
In view of the higher order H\"older estimate given by Lemma~\ref{boundschauder}~(i), we know that $g^R$ uniformly converges to $g$ in $\Ckin^2(K)$, up to a subsequence. 
Besides, due to the estimate derived in \eqref{barrierinitial}, the limiting function $g$ is a solution to~\eqref{gnk} that matches the initial data $g_{\rm in}$ continuously.

As for \eqref{l2l1}, we notice that the function $\big(h^R-M_1\big)_\pm$ with $h^R:=\m^{-\frac{1}{2}}g^R$ verifies 
\begin{equation*}
\left(\p_t+v\cdot\nabla_x\right)\big(h^R-M_1\big)_\pm \le \rfp_{h^R}\,\lfp \big(h^R-M_1\big)_\pm  {\quad \rm in\ } (0,T]\times[-R,R]^d\times\R^d. 
\end{equation*} 
Integrating the equation against the function $\int_{[-R,R]^d}\big(h^R-M_1\big)_\pm\dif x$ yields that 
\begin{align*}
\int_{[-R,R]^d\times\R^d}\big(h^R(t,\cdot,\cdot)-M_1\big)_\pm\dif m -\int_{[-R,R]^d\times\R^d}\big(h^R(0,\cdot,\cdot)-M_1\big)_\pm\dif m\le0. 
\end{align*}
Sending $R\rightarrow\infty$, we acquire
\begin{equation*}
\|(h-M_1)_\pm\|_{L_t^\infty([0,T]; L^1(\R^d\times\R^d,\,\dif m))}
\le \|(\hin-M_1)_\pm\|_{L^1(\R^d\times\R^d,\,\dif m)}, 
\end{equation*}
which implies the estimate \eqref{l2l1} as asserted. 
The proof is complete. 
\end{proof}

The following proposition concerned with the uniqueness of the Cauchy problem~\eqref{gnk} is derived from a Grönwall-type argument. 
The standard scaling technique and the H\"older estimate up to the initial time given by Lemma~\ref{boundschauder}~(ii) can improve the integrability with respect to $t$ in the energy estimate so that Grönwall's inequality becomes admissible; see \eqref{gronwall} below for the precise expression. This kind of phenomena was also noticed in \cite{HST1} (see the remarks in \S1.4.2). 
Besides, the global energy estimate of the equation~\eqref{gnk} is not available when the spatial domain is unbounded, since there is no decay of the solution as $|x|\rightarrow\infty$. 
To work it out, we take advantage of the idea originated from the uniformly local space that was used in \cite{Kato,HST2}. We also remark that such a technique is not necessary when working with the periodic box $\T^d$.

\begin{proposition}[Uniqueness]\label{unique}
Let the domain $\O_x=\T^d$ or $\R^d$, the constants $\a_0\in(0,1)$, and the functions $0\le g_1,g_2\lesssim\m^\frac{1}{2}$ be two solutions to \eqref{gnk} in $(0,T]\times\O_x\times\R^d$ associated with the same initial data $g_{\rm in}\in {\mathcal C}^{\a_0}(\O_x\times\R^d)$ such that 
\begin{equation*}
\int_{\R^{d}} g_{\rm in}\m^{\frac{1}{2}}\dif v\ge\lambda {\ \ \rm in\ }\O_x{\quad\rm and\quad}
0\le g_{\rm in}\lesssim\m^\frac{1}{2}{\ \ \rm in\ }\O_x\times\R^d.
\end{equation*} 
Then,  $g_1=g_2$ in $[0,T]\times\O_x\times\R^d$. 
\end{proposition}

\begin{proof}
In view of the lower bound given by Lemma~\ref{lowerlemma1} and Proposition~\ref{lower}, we know that there is some constant $\lambda_*\in(0,1)$ depending only on universal constants, $T$ and the initial data such that
\begin{equation}\label{uniquelower}
\int_{\R^{d}} g_i\m^{\frac{1}{2}}\dif v\ge\lambda_* {\ \ \rm in\ } [0,T]\times\O_x, \quad i=1,2.
\end{equation} 
Therefore, we may assume $T=\Lambda^{-1}$ with $\Lambda>1$. 
Let us set the difference $\gg:=e^{-\frac{|v|^2t}{8}}(g_1-g_2)$. We have to show that $\gg$ is identically zero. 

In view of the equation~\eqref{gnk}, a direct computation yields that the function $\gg$ satisfies 
\begin{equation}\label{uniqueiden0}
\begin{split}
&\left(\p_t+v\cdot\nabla_x\right)\gg +\frac{|v|^2}{8}\gg\\
&=e^{-\frac{|v|^2t}{8}}\left(\RO[g_1]-\RO[g_2]\right)\UO[g_1]
+\RO[g_2]\left(\UO[\gg] +\frac{t}{2}\,v\cdot\nabla_v\,\gg +\left(\frac{dt}{4}+\frac{|v|^2t^2}{16}\right)\gg\right), 
\end{split}
\end{equation}
with the initial condition $\gg(0,x,v)=0$ in $\O_x\times\R^d$. 

Let $y\in\R^d$. We introduce a cut-off function $\phi_y(x):=\phi(x-y)$, where $\phi\in C^\infty_c(\R^d)$ is valued in $[0,1]$ such that $\phi|_{B_1}\equiv 1$, $\phi|_{B_2^c}\equiv 0$ and $|\nabla\phi|\lesssim1$ in $\R^d$. 
For any $t\in(0,T]$, integrating the equation~\eqref{uniqueiden0} against $\phi_y^2\gg$ in $\O_x\times\R^d$ and applying integration by parts yields that 
\begin{align*}
\frac{1}{2}\int_{\O_x\times\R^d}\phi_y^2\gg^2(t)
=\int_0^t\!\int_{\O_x\times\R^d} \bigg\{(v\cdot\nabla\phi_y)\,\phi_y\gg^2 -\frac{|v|^2}{8}\phi_y^2\gg^2
 + e^{-\frac{|v|^2t}{8}}\left(\RO[g_1]-\RO[g_2]\right)\UO[g_1]\,\phi^2\gg &\nonumber\\
-\RO[g_2]\left(\left|\nabla_{v}\left(\m^{-\frac{1}{2}}\,\gg\right)\right|^2\!\mu\phi_y^2
+\frac{dt}{4}\phi_y^2\gg^2
-\left(\frac{dt}{4}+\frac{|v|^2t^2}{16}\right)\phi_y^2\gg^2\right)\bigg\}&.
\end{align*}
Since $\RO[g_2]\in[0,\Lambda]$, for any $t\in(0,T]$, we have 
\begin{equation*}\label{uniqueiden}
\frac{1}{2}\int_{\O_x\times\R^d}\phi_y^2\gg^2(t)
\le \int_0^t\!\int_{\O_x\times\R^d} \bigg\{|v||\nabla\phi_y|\,\phi_y\gg^2 -\frac{|v|^2}{16}\phi_y^2\gg^2
+ \m^{-\frac{1}{4}}\left|\RO[g_1]-\RO[g_2]\right||\UO[g_1]|\phi_y^2|\gg|\bigg\}. 
\end{equation*}
Due to the elementary inequality $|z^\beta-1|\le|z-1|$ $(z\in\R^+)$ and the lower bound estimate in \eqref{uniquelower}, as well as the boundedness assumption on $g_1,g_2$, we have 
\begin{align*}
\left|\RO[g_1]-\RO[g_2]\right|\le \RO[g_1]^\frac{\b-1}{\b}|\RO[\gg]|^\frac{1}{\b}\le \frac{1}{\lambda_*}\int_{\R^d}|\gg(t,x,\cdot)|\m^\frac{1}{2}
\lesssim_{\lambda_*}1  {\quad\rm in\ }[0,1]\times\O_x. 
\end{align*}
It then follows that for any $t\in(0,T]$, 
\begin{align}\label{uniquenon0}
\frac{1}{2}\int_{\O_x\times\R^d}\phi_y^2\gg^2(t)
\le& \int_0^t\!\int_{\O_x\times\R^d} \left(|v||\nabla\phi_y|\,\phi_y\gg^2 -\frac{|v|^2}{16}\phi_y^2\gg^2\right)\nonumber\nonumber\\
  &+\frac{1}{\lambda_*}\int_0^t\big\|\mu^{-\frac{3}{8}}\UO[g_1]\big\|_{L^\infty_{x,v}}\int_{\O_x\times\R_v^d} \phi_y^2|\gg(t,x,v)|\mu^{\frac{1}{8}}
   \int_{\R_\xi^d}|\gg(t,x,\xi)|\m^\frac{1}{2}\dif\xi \nonumber\\  
\lesssim&_{\lambda_*} \int_0^t\!\int_{\O_x\times\R^d} |\nabla\phi_y|^2\gg^2 + 
\int_0^t\big\|\mu^{-\frac{3}{8}}\UO[g_1]\big\|_{L^\infty_{x,v}}\int_{\O_x\times\R^d} \phi_y^2\gg^2,   
\end{align}
where we used the Cauchy–Schwarz inequality and Hölder's inequality in the last line. 
Recalling that $\phi_y(x)=\phi(x-y)\in C^\infty_c(\R^d)$ and $|\nabla\phi|\lesssim1$ in $\R^d$, we have 
\begin{align*}
\sup_{y\in\R^d}\int_{\O_x\times\R^d}|\nabla\phi_y|^2\gg^2
\lesssim \sup_{y\in\R^d}\int_{\O_x\times\R^d} \phi_y^2\gg^2. 
\end{align*}
By the definition of $\UO[g_1]$ and the upper bound $g_1\lesssim \m^\frac{1}{2}$, 
\begin{align*}
\big\|\mu^{-\frac{3}{8}}\UO[g_1]\big\|_{L^\infty_{x,v}}
\lesssim 1+\big\|\mu^{-\frac{3}{8}}\Delta_{v}g_1\big\|_{L^\infty_{x,v}}. 
\end{align*}
Hence, for any $t\in(0,T]$, taking supremum over $y\in\R^d$ in \eqref{uniquenon0}, we obtain 
\begin{align}\label{uniquenon}
\sup_{y\in\R^d}\int_{\O_x\times\R^d}\phi_y^2\gg^2(t)
\lesssim_{\lambda_*} \int_0^t\left(1+\big\|\mu^{-\frac{3}{8}}\Delta_{v}g_1\big\|_{L^\infty_{x,v}}\right) \sup_{y\in\R^d}\int_{\O_x\times\R^d} \phi_y^2\gg^2. 
\end{align}

Now we have to consider the pointwise estimate on $D_v^2g_1$. Let $z_0=(t_0,x_0,v_0)\in(0,T]\times\O_x\times\R^d$ and $2r=t_0^\frac{1}{2}$. 
In view of \eqref{uniquelower}, Lemma~\ref{boundschauder}~(ii) implies that there exists some constant ${\a_*}\in(0,1)$ with the same dependence as $\lambda_*$ such that 
\begin{align*}
\|g_1\|_{\Ck^{\a_*}\left([0,T]\times\O_x\times B_1(v_0)\right)}
\lesssim_{\lambda_*} 1+[g_{\rm in}]_{{\mathcal C}^{\a_0}(\O_x\times\R^d)}. 
\end{align*}
Then, applying the interior Schauder estimate (Proposition~\ref{interiorschauder}) and the upper bound $g_1\lesssim \m^\frac{1}{2}$ yields that 
\begin{equation*}
\begin{split}
\big\|D_v^2g_1\big\|_{L^\infty(Q_{r}(z_0))}
&\lesssim_{\lambda_*} r^{-2}\| g_1-g_1(z_0)\|_{L^\infty(Q_{2r}(z_0))}
+r^{\a_*}\left[\RO[g_1]\left(\frac{d}{2}-\frac{|v|^{2}}{4}\right)g_1\right]_{\Ck^{\a_*}\left(Q_{2r}(z_0)\right)}\\
&\lesssim_{\lambda_*} r^{-2+\frac{{\a_*}}{4}}\m^\frac{3}{8}(v_0)[g_1]_{\Ck^{\a_*}(Q_{2r}(z_0))}^\frac{1}{4}
 +\m^\frac{3}{8}(v_0)[g_1]_{\Ck^{\a_*}(Q_{2r}(z_0))}^\frac{1}{4}\\
& \lesssim_{\lambda_*}t_0^{-1+\frac{\a_*}{8}} \m^\frac{3}{8}(v_0) \Big(1+[g_{\rm in}]_{{\mathcal C}^{\a_0}(\O_x\times\R^d)}^\frac{1}{4}\Big). 
\end{split}
\end{equation*}
By the arbitrariness of $z_0$, we know that for any $s\in(0,T]$,  
\begin{equation*}
\big\|\mu^{-\frac{3}{8}}\Delta_{v}g_1(s)\big\|_{L^\infty(\O_x\times\R^d)}
\lesssim_{\lambda_*} \Big(1+[g_{\rm in}]^\frac{1}{4}_{{\mathcal C}^{\a_0}(\O_x\times\R^d)}\Big)\,s^{-1+\frac{\a_*}{8}}. 
\end{equation*} 
Dragging this estimate into \eqref{uniquenon} yields that for any $t\in(0,T]$,
\begin{align}\label{gronwall}
\sup_{y\in\R^d}\int_{\O_x\times\R^d}\phi_y^2\gg^2(t)
\le C_*\int_0^t\dif s\left(1+s^{-1+\frac{\a_*}{8}}\right) \sup_{y\in\R^d}\int_{\O_x\times\R^d}\phi_y^2\gg^2(s),  
\end{align}
where the constant $C_*>0$ depends only on universal constants and the initial data. 
The desired result is then given by Grönwall's inequality. 
\end{proof}

\subsection{Global regularity}\label{smoothsection}
The instantaneous smoothness a priori estimate in Theorem~\ref{wellpose}~(i) is made up of the lower bound 
given by Proposition~\ref{lower} and the following proposition. 

\begin{proposition}\label{smoothness}
Let $\O_x=\T^d$ or $\R^d$,  
$\underline{T}\in(0,T)$, and the function$g$ be a solution to \eqref{gnk} in $(0,T)\times\T^d\times\R^d$ such that 
\begin{equation}\label{lowerunif}
\RO[g]\ge\lambda {\ \ \rm in\ }\left[{\underline{T}}/{4},T\right]\times\O_x {\quad\rm and\quad}
	0\le g\le\Lambda\m^\frac{1}{2}{\ \ \rm in\ }[0,T]\times\O_x\times\R^d. 
\end{equation}
Then, for any $\n\in\big(0,\frac{1}{2}\big)$ and $k\in\N$, we have 
\begin{equation*}
	\|\m^{-\n}g\|_{{\mathcal C}^{k}\left([\underline{T}, T]\times\O_x\times\R^d\right)} \le C_{\underline{T},\n,k}, 
\end{equation*}
for some constant $C_{\underline{T},\n,k}>0$ depending only on universal constants, $\underline{T},\n$ and $k$. 
\end{proposition}

Generally speaking, if $g$ be a solution to \eqref{gnk} in $(0,T]\times\O_x\times\R^d$ constructed by Proposition~\ref{existence} ($\O_x=\T^d$) or Corollary~\ref{existwhole} ($\O_x=\R^d$), then the uniform positivity  assumption \eqref{lowerunif} should be replaced by 
\begin{equation*} 
\RO[g]\ge\lambda_{t,x} {\ \ \rm in \ } (0,T]\times\O_x
{\quad\rm and\quad} 0\le g\le\Lambda\m^\frac{1}{2}	{\ \ \rm in \ } \O_x\times\R^d,
\end{equation*}
where $\lambda_{t,x}>0$ may degenerate to zero as $t\rightarrow0$, or $t+|x|\rightarrow\infty$; see Proposition~\ref{lower} and Remark~\ref{lowerwholeremark}. 
As an immediate consequence of the above proposition, for any $\n\in\big(0,\frac{1}{2}\big)$, $k\in\N$ and for any compact subset $K\subset(0,T]\times\O_x$, there exists some constant $C_{\n,k,K}>0$ depending only on universal constants, $\n,k$ and $K$ such that 
\begin{equation*}
	\|\m^{-\n}g\|_{{\mathcal C}^{k}\left(K\times\R^d\right)} \le C_{\n,k,K}, 
\end{equation*}
which is exactly the assertion in Theorem~\ref{wellpose} (i).

In order to show the higher regularity, we will apply the bootstrap procedure developed in \cite{IS} which was intended for the non-cutoff Boltzmann equation. 
The classical bootstrap iteration proceeds by differentiating the equation, using a priori estimates to the new equation to improve the regularity of solutions, and repeating such procedure. 
Nevertheless, since $\Ck^{2+\a}\not\subset{\mathcal C}_x^1$ for any $\a\in(0,1)$ by their definitions, the hypoelliptic structure of the equation~\eqref{gnk} lacks gain of enough regularity in $x$-variable which disables the $x$-differentiation at each iteration. 
Indeed, the Schauder type estimate provided by Lemma~\ref{boundschauder}~(i) only shows that the solution to \eqref{gnk} belongs to ${\mathcal C}^\frac{2+\a}{3}$ with respect to $x$-variable. 
In order to overcome it, we have to apply estimates to increments of the solution to recover a full derivative. 
From now on, for $y \in \R^{d}$ and $w\in\R\times\R^d\times\R^d$, we denote the spatial increment 
\begin{equation*}
\d_yg(z):= g(w\circ (0,y,0))-g(w).
\end{equation*}

Let us proceed with the proof of the regularity estimate. 
\begin{proof}[Proof of Proposition~\ref{smoothness}]
We are going to show that for any multi-index $k:=(k_t,k_x, k_v)\in\N\times\N^d\times\N^d$ and $\n\in\big(0,\frac{1}{2}\big)$, 
there exists some constant $\a_k\in(0,1)$ depending only on $|k|$ such that for any $Q_{r}(z_0)\subset[\underline{T}/2,T] \times \O_x\times\R^d$, 
\begin{equation}\label{inequality}
\|\p_{t}^{k_t}\p_{x}^{k_x}\p_{v}^{k_v}g\|_{\Ck^{2+\a_k}(Q_r(z_{0}))} 
\lesssim_{\underline{T},\n,k} \mu^\n(v_0). 
\end{equation}

For simplicity, we will omit the domain in estimates below, since the estimates can be always localized around the center $z_0$. 
		
\medskip\noindent{\textbf{Step 0.}} The case of $k=(0,0,0)$ in \eqref{inequality} is a direct consequence of Lemma \ref{boundschauder}~(i).

\medskip\noindent{\textbf{Step 1.}} We will establish that \eqref{inequality} holds for any differential operators of the type $\p^{k_x}_x$. It suffices to show that for any $n\in\N$, $k_x\in\N^d$ with $|k_x|=n$, $\n\in\big(0,\frac{1}{2}\big)$ and $y \in B_{\frac{r^3}{4}}$,
\begin{equation}\label{xincre}
\|\delta_y\p^{k_x}_xg\|_{\Ck^{2+\a_n}} \lesssim_{\underline{T},\n,n} |y|\mu^{\nu}(v_0).
\end{equation}
Indeed, sending $y \rightarrow 0$ in \eqref{xincre} will complete this step. 
		
Based on an induction on $|k_x|=n$, we suppose that \eqref{xincre} holds for any $|k_x|\le n-1$, which implies for any $k_x\in\N^d$ with $|k_x|\le n$, 
\begin{equation}\label{step1hypo}
\|\p^{k_x}_xg\|_{\Ck^{2+\a_n}} \lesssim_{\underline{T},\n,n} \mu^{\nu}(v_0).
\end{equation}
We remark that the induction here begin with \eqref{step1hypo} for $|k_x|=0$, which holds due to the previous step. 
		
Let $q:=\delta_{y}\p^{k_x}_x g$ with $|k_x|=n$. 
Lemma~\ref{incrementi-IS} and \eqref{step1hypo} gives 
\begin{equation}\label{step1lemma}
\|q\|_{\Ck^{\a_n}} 
\lesssim \|\p^{k_x}_x g\|_{\Ck^{2+\a_n}}\|(0,y,0)\|^{2}
\lesssim_{\underline{T},\n,n} |y|^{\frac{2}{3}}\mu^{\nu}(v_0).
\end{equation}
Therefore, we have to enhance the exponent $\frac{2}{3}$ on the right hand side to $1$; as a sacrifice, the H\"older exponent on the left hand side will decrease. 
		
Set $\t_yg(w):=g(w\circ (0,y,0))$ for $y\in\R^d$ and $w\in\R\times\R^d\times\R^d$. 
A direct computation shows that $q$ verifies the equation, 
\begin{equation}\label{step1eq}
(\p_{t} + v \cdot \nabla_{x}) q = \RO[g]\, \UO[q] 
+ \sum\nolimits_{|i| \le n \atop i \le k_{x}} \delta_{y}\hat{D_{i}}\RO[g]\, \UO[\t_yD_{i} g]
+\sum\nolimits_{|i| \le n-1 \atop i \le k_{x}}\hat{D_{i}}\RO[g]\, \UO[\delta_{y} D_{i} g],
\end{equation}
where the multi-indices such that $i \le k$ means each component of $i$ is lower or equal than the corresponding component of $k$
and $\hat{D_{i}}$ denotes the differential operator satisfying $\p_{t}^{k_t} D_{x}^{k_{x}} = \hat{D_{i}} \circ D_{i}$. 
		
In view of \eqref{step1hypo} \eqref{step1lemma} and the induction hypothesis, each term in two summations on the right hand side of \eqref{step1eq} is bounded in $\Ck^{\a_n}$ by $C_n\|(0,y,0)\|^2\mu^{\nu'}(v_0)$ for any $\nu'\in(0,\nu)$. 
Then, by the interior Schauder estimate (Proposition~\ref{interiorschauder}), 
\begin{equation}\label{gain-x}
\|q\|_{\Ck^{2+\a_n}}\lesssim_{\underline{T},\n',n}\|(0,y,0)\|^2\mu^{\nu'}(v_0).
\end{equation}
Combining Lemma~\ref{mio-lemma} with \eqref{step1lemma} and \eqref{gain-x}, we obtain \eqref{xincre}.

\medskip\noindent{\textbf{Step 2.}} For \eqref{inequality} in the case of $k_v=0$, we proceed a bidimensional induction on $(m,n)=(k_t,|k_{x}|)$ such that for any $\n\in\big(0,\frac{1}{2}\big)$, 
\begin{equation}\label{inductive-t}
\| \p_{t}^{k_{t}} D^{k_{x}}_{x} g \|_{\Ck^{2+\a_{m,n}}} \lesssim_{\underline{T},\n,m,n}\mu^{\nu}(v_0).
\end{equation}
Based on the previous step ($m=0$), we have to show that \eqref{inductive-t} holds for $k_t=m\ge1$, $|k_{x}| =n$, under the induction hypothesis that \eqref{inductive-t} holds for any $k_t\le m-1$ and $|k_x|\le n+1$, 
		
With $k_t=m>0$, $|k_{x}| =n$, set $q:=\p_{t}^{k_t} D_{x}^{k_{x}} g$. Then, there holds
\begin{equation}\label{eqstep2}
(\p_{t} + v \cdot \nabla_{x}) q = \RO[g]\, \UO[q] + \sum\nolimits_{i \le (k_t, k_{x}, 0) \atop i \ne (k_t, k_{x}, 0)}\hat{D_{i}}\RO[g]\, \UO[D_{i} g],
\end{equation}
where we use the notation $\hat{D_{i}}$ such that $\p_{t}^{k_t} D_{x}^{k_{x}} = \hat{D_{i}} \circ D_{i}$. 
		
By the induction hypothesis, each term in the remainder (the summation on the right hand side of \eqref{eqstep2}) with $i \ne (0,0,0)$ can be controlled in $\Ck^{\a_{m,n}}$. 
It now suffices to deal with the exceptional term $\p_{t}^{k_t} D_{x}^{k_{x}} \RO[g]$ so that the whole remainder can be controlled in $\Ck^{\a_{m,n}}$; and then \eqref{inductive-t} follows from the interior Schauder estimate (Proposition~\ref{interiorschauder}). 
To this end, using Lemma~\ref{kinetic-degree} and the induction hypothesis with the pair $(m-1,n)$ yields 
\begin{equation}\label{1}
\| (\p_{t} + v \cdot \nabla_{x}) \p_{t}^{m-1} D^{k_{x}}_{x} g \|_{\Ck^{\a_{m,n}}} 
\lesssim_{\underline{T},\n,m,n}\mu^{\n}(v_0).
\end{equation}
Due to the induction hypothesis with the pair $(m-1,n+1)$, for any $\nu'\in(0,\nu)$, 
\begin{align}\label{2}
\mu^{-\nu'}(v_0)\| (v \cdot \nabla_{x}) \p_{t}^{m-1} D^{k_{x}}_{x} g \|_{\Ck^{2+\a_{m,n}}} 
\lesssim_{\n,\n'} \mu^{-\nu}(v_0)\| \p_{t}^{m-1} \nabla_{x} D^{k_{x}}_{x} g \|_{\Ck^{2+\a_{m,n}}} 
\lesssim_{\underline{T},\n,m,n} 1.
\end{align}
Then, \eqref{1} and \eqref{2} produce the bound on $\mu^{-\nu'}(v_0)\|q\|_{\Ck^{\a_{m,n}}}$.

\medskip\noindent{\textbf{Step 3.}} Similarly, to show \eqref{inequality} for any differential operator $\p_{t}^{k_{t}} D^{k_{x}}_{x} D^{k_{v}}_{v}$, 
we proceed a bidimensional induction on $(m,n)=(k_t+|k_x|,k_{v})$ such that for any $\n\in\big(0,\frac{1}{2}\big)$,
\begin{equation}\label{inductive-v}
\| \p_{t}^{k_{t}} D^{k_{x}}_{x} D^{k_{v}}_{v} g \|_{\Ck^{2+\a_{m,n}}} 
\lesssim_{\underline{T},\n,m,n}\mu^{\nu}(v_0).
\end{equation}
The case $n=0$ is treated in the previous step. By Lemma~\ref{kinetic-degree} and the induction hypothesis \eqref{inductive-v} with $k_{t} + |k_{x}| = m$ and $|k_v|=n-1$, $n \ge 1$, we have
\begin{equation*}
\| \p_v  \p_t^{k_{t}} \p_x^{k_{x}} \p_v^{k_v} g \|_{\Ck^{\a_{m,n}}} 
\lesssim \| \p_t^{k_{t}} \p_x^{k_{x}} \p_v^{k_v} g \|_{\Ck^{1+\a_{m,n}}}  
\lesssim_{\underline{T},\n,m,n} \m^{\n}(v_0).
\end{equation*}
Computing the equation satisfied by $\p_v \p_{t}^{k_{t}} \p^{k_{x}}_{x} \p^{n-1}_{v} g$ and proceeding the argument as in the previous step, we conclude the proof.
\end{proof}

\section{Diffusion asymptotics}\label{diffusionlimit}
This section is devoted to the study of the global in time quantitative diffusion asymptotics which consists of the (uniform-in-$\e$) convergence towards the equilibrium over long times and of the finite time asymptotics, including the results of Theorem~\ref{wellpose}~(ii) and Theorem~\ref{limit}. 

First of all, let us introduce the required notation. 
For any scalar or vector valued function $\Psi\in L^1(\R^d,\dif\mu)$, we denote its velocity mean by
	\begin{equation*}
	\la \Psi\ra:=\int_{\mathbb{R}^d} \Psi(v) \dif\mu. 
	\end{equation*}
For any couple of functions (scalar, vectors or $d\times d$-matrices) $\Psi_1,\Psi_2\in L^2\left(\T^d\times\mathbb{R}^d,\dif m\right)$, we denote their $L^2$ inner product with respect to the measure $\dif m$ by
\begin{equation*}
\left(\Psi_1, \Psi_2\right):=\int_{\T^d\times\mathbb{R}^d} \Psi_1(x,v)\Psi_2(x,v)\dif m, 
\end{equation*}
where the multiplication between the couple in the integrand is replaced by scalar contraction product, if $\Psi_1,\Psi_2$ is a couple of vectors or matrices.  
	
Recalling our notation for the Ornstein-Uhlenbeck operator  $\lfp=\left(\nabla_v-v\right)\cdot\nabla_v$, we apply the substitutions $f_\e=\m h_\e$, $f_{\e,\rm in}=\m\hine$ to \eqref{epfnk} and obtain 
\begin{equation}\label{epnk}
\left\{ 
\begin{aligned}
\ &\left(\e\p_t+v\cdot\nabla_x\right) h_\e(t,x,v) = \frac{1}{\e}\la h_\e\ra^\b(t,x)\,\lfp h_\e(t,x,v), \\
\ &\;h_\e(0,x,v) = h_{\e,\rm in} (x,v), \\
\end{aligned}
\right. 
\end{equation} 
In this setting, by applying integration by parts, for any $h_1,h_2\in  \mathcal{C}_c^\infty(\T^d\times\R^d)$ we get
\begin{equation*}
\left(h_1,\, \lfp h_2\right)=-\left(\nabla_vh_1,\, \nabla_vh_2\right). 
\end{equation*}
We will use this identity repeatedly in the computation below. 	
Then, the operator $\lfp$ is self-adjoint with respect to the inner product $(\cdot,\cdot)$ and the bracket $\la \cdot\ra$ is a projection on the null space of $\lfp$. 
Moreover, as the total mass is conserved, we define
	\begin{equation}\label{FP2cond}
	M_0:=\int_{\T^d\times\mathbb{R}^d} h_\e\dif m=\int_{\T^d}\la h_\e\ra\dif x.
	\end{equation}
 Proceeding with the macro-micro (fluid-kinetic) decomposition, we define the orthogonal complement of the projection~$\la \cdot\ra$ of $h_\e$ as 
	$$h^\perp_\e(t,x,v):=h_\e(t,x,v)-\la h_\e\ra(t,x).$$ 
	In this framework, the local mass $\la h_\e\ra$ is the macroscopic (fluid) part and the complement $h^\perp_\e$ is the microscopic (kinetic) part. 
Besides, taking the bracket $\la\cdot\ra$ after multiplying the equation in \eqref{epnk} with $1$ and $v$ leads to the following macroscopic equations, 
	\begin{align}
	&\e\partial_t\la h_\e\ra+\nabla_x\cdot\la vh_\e\ra=0,  \label{hydroFP1} \\
	&\e\partial_t\la vh_\e\ra + \nabla_x\cdot\la v^{\otimes2}h_\e\ra
	=-\frac{1}{\e}\la h_\e\ra^\beta \la vh_\e\ra,   \label{hydroFP2}
	\end{align}
	where $\la vh_\e\ra$ and $\la v^{\otimes2}h_\e\ra$ represent the local momentum and the stress tensor, respectively.

\subsection{Long time behavior}\label{longtimesection}
Our aim is to establish the (uniform-in-$\e$) exponential decay towards the equilibrium $M_0$ for \eqref{epnk}. 
In particular, when $\e=1$, it sets up the exponential convergence in each order derivative based on the smoothness a priori estimates given in Subsection~\ref{smoothsection}. 
	
We remark that the classical coercive method is not applicable in our case to obtain the convergence to equilibrium due to the degeneracy of the ellipticity of the spatially inhomogeneous equation. Indeed, the Poincaré inequality only produces a spectral gap on the orthogonal complement of the projection $\la\cdot\ra$; see \eqref{hypoestimate0} below. 
As we mentioned in Subsection~\ref{background}, there are several ways to achieve the long time asymptotics. 
We will mainly follow the argument presented in \cite{EGKM} (see also \cite{KGH}) in a simpler scenario. 
It would also allow us to see some similarity among \cite{EGKM,DMS,He}.

\begin{proposition}\label{longtime}
Let the function $\lambda_t:\R_+\rightarrow[0,\Lambda]$ with the derivative $\lambda_t'\le0$ on $\R_+$. 
If $h_\e$ is a solution to \eqref{epnk} in $\R_+\times\T^d\times\R^d$, associated with the initial data $0\le\hine\le\Lambda$, satisfying 
\begin{equation}\label{hp-prop}
\la h_\e\ra^\beta(t,x)\ge\lambda_t {\ \ \rm in\ }\R_+\times\T^d {\quad\rm and\quad}
\int_{\R_+}\left(\lambda_t+\lambda_t'\right)\dif t=\infty,  
\end{equation}
then the solution $h_\e$ converges to the state $M_0$ in $L^2(\dif m)$ as $t\rightarrow\infty$; more precisely, there exists some universal constant $c>0$ such that for any $t>0$, we have 
\begin{equation}\label{longtimeestimate}
\|h_\e(t,\cdot,\cdot)-M_0\|^2_{L^2(\dif m)}
\lesssim \| \hine-M_0\|^2_{L^2(\dif m)}\exp\left(-c\int_0^t\left(\lambda_s+\lambda_s'\right)\dif s\right).  
\end{equation}
\end{proposition}

\begin{proof}
Since the velocity mean of the microscopic part vanishes, $\la h^\perp_\e\ra=0$, using the equation~\eqref{epnk} and the Poincaré inequality yields that
\begin{align}\label{hypoestimate0}
\frac{1}{2}\frac{\dif}{\dif t}\| h_\e-M_0\|^2_{L^2(\dif m)}
&=\frac{1}{\e^2} \left(\la h_\e\ra^\b\lfp h_\e,\, h_\e-M_0\right) =-\frac{1}{\e^2}\left(\la h_\e\ra^\b\nabla_vh^\perp_\e,\, \nabla_vh^\perp_\e\right)\nonumber\\
&\le-\frac{\lambda_t}{\e^2}\| \nabla_v h^\perp_\e\|^2_{L^2(\dif m)}
\lesssim -\frac{\lambda_t}{\e^2}\| h^\perp_\e\|^2_{L^2(\dif m)}. 
\end{align}
Now we have to recover a new entropy that would give some bound on the projection~$\la h_\e\ra-M_0$. 
		
For every test function $v\cdot\Psi(t,x)\mu$, with a vector-valued function $\Psi\in H^1_{t,x}(\R^{+}\times\T^d,\R^d)$, we write the weak formulation of \eqref{epnk} as follows	
\begin{equation*}
\frac{\dif}{\dif t}\left(v\cdot\Psi,\, h_\e\right)=\frac{1}{\e} \left( v^{\otimes2}:\nabla_x\Psi,\, h_\e\right) +\left(v\cdot\partial_t\Psi,\, h_\e\right) 
+\frac{1}{\e^2}\left(\la h_\e\ra^\beta \lfp v\cdot\Psi,\, h_\e\right). 
\end{equation*}
Taking the macro-micro decomposition into account, we obtain from the above expression 
\begin{equation}\label{FPpert2}
\begin{split}
\frac{\dif}{\dif t}\left(v\cdot\Psi,\, h^\perp_\e\right)
=&\frac{1}{\e}\left(|v_1|^2\trace(\nabla_x\Psi),\, \la h_\e\ra-M_0\right)+\frac{1}{\e}\left(v^{\otimes2}:\nabla_x\Psi,\, h^\perp_\e\right)\\
&+\left(v\cdot\partial_t\Psi,\, h^\perp_\e\right)
-\frac{1}{\e^2}\left(\la h_\e\ra^\beta v\cdot\Psi,\, h^\perp_\e\right). 
\end{split}
\end{equation}
Let us now introduce an auxiliary function $u(t,x)$: for any fixed $t\in\R_+$, $u(t,x)$ is defined as the solution of the following elliptic equation under the compatibility condition \eqref{FP2cond}, 
\begin{equation}\label{hypoPoisson0} 
	-\Delta_x u=\la h_\e\ra-M_0 {\quad\rm in\ }\T^d, 
\end{equation}
whose elliptic estimate states 
\begin{equation}\label{hypoPoisson1}
\|\nabla_x u\|_{L^2_x}+\|\nabla^2_x u\|_{L^2_x}
\lesssim \|\la h_\e\ra-M_0\|_{L^2_x}. 
\end{equation}
Besides, observing that $\la vh_\e\ra$=$\la vh^\perp_\e\ra$, from \eqref{hydroFP1}, we get 
\begin{equation*}
\e\partial_t\la h_\e\ra+\nabla_x\cdot\la vh^\perp_\e\ra=0. 
\end{equation*}
Combining this macroscopic relation with \eqref{hypoPoisson0}, we have
\begin{equation*}
\int_{\T^d}|\nabla_x\left(\partial_tu\right)|^2
		=\int_{\T^d}\partial_tu\,\partial_t\la h_\e\ra
		=-\frac{1}{\e}\int_{\T^d}\partial_tu\,\nabla_x\cdot\la vh^\perp_\e\ra
		=\frac{1}{\e}\int_{\T^d}\nabla_x\left(\partial_tu\right)\cdot\la vh^\perp_\e\ra. 
\end{equation*}
It then follows from H\"older's inequality that 
\begin{equation}\label{hypoPoisson2}
\|\nabla_x(\p_tu)\|_{L^2_x}
\le\frac{1}{\e}\|\la vh^\perp_\e\ra\|_{L^2_x}
\lesssim \frac{1}{\e}\| h^\perp_\e\|_{L^2(\dif m)}. 
\end{equation}
Choosing $\Psi=\nabla_xu$ in \eqref{FPpert2} yields
\begin{equation*}\label{recover}
\begin{split}
-\frac{1}{\e}\left(|v_1|^2\Delta_xu,\la h_\e\ra-M_0\right)
&+\frac{\dif}{\dif t}\left(v\cdot\nabla_xu,\, h^\perp_\e\right)\\
&\lesssim \left(\frac{1}{\e}\| \nabla^2_xu\|_{L^2_x}
+\|\nabla_x\left(\partial_tu\right)\|_{L^2_x}
+\frac{1}{\e^2}\|\nabla_xu\|_{L^2_x}\right)\| h^\perp_\e\|_{L^2(\dif m)}.
\end{split}
\end{equation*}
Applying \eqref{hypoPoisson0}, \eqref{hypoPoisson1} and \eqref{hypoPoisson2}, we have
\begin{equation*}
\begin{split}
\frac{1}{\e}\|\la h_\e\ra-M_0\|^2_{L^2_x}
+\frac{\dif}{\dif t}\left(v\cdot\nabla_xu,\, h^\perp_\e\right)
&\lesssim \frac{1}{\e^2}\|\la h_\e\ra-M_0\|_{L^2_x}\| h^\perp_\e\|_{L^2(\dif m)}
+\frac{1}{\e}\| h^\perp_\e\|_{L^2(\dif m)}^2.
\end{split}
\end{equation*}
By the Cauchy–Schwarz inequality, we arrive at 
\begin{equation}\label{necessario}
\|\la h_\e\ra-M_0\|^2_{L^2_x}
	+\e\frac{\dif}{\dif t}\left(v\cdot\nabla_xu,\, h^\perp_\e\right)
\lesssim \frac{1}{\e^2}\| h^\perp_\e\|_{L^2(\dif m)}^2.
\end{equation}

Then, \eqref{necessario} combined with \eqref{hypoestimate0} implies that 
\begin{align*}
\frac{\dif}{\dif t}\E_\e(t)
&\lesssim -\frac{1-\d}{\e^2}\| h^\perp_\e\|_{L^2(\dif m)}^2
  -{\d\lambda_t}\|\la h_\e\ra-M_0\|^2_{L^2_x}
  +\d\e\lambda_t'\left(v\cdot\nabla_xu,\, h^\perp_\e\right)\\
&\le -\d\lambda_t\| h_\e-M_0\|^2_{L^2(\dif m)} 
  -\d\lambda_t'\left|\left(v\cdot\nabla_xu,\, h^\perp_\e\right)\right|,
\end{align*}
where the constant $\delta\in\big(0,\frac{1}{2}\big)$ will be determined and the modified entropy $\E_\e$ is defined by 
$$\E_\e(t):= \| h_\e-M_0\|^2_{L^2(\dif m)} +{\d\e\lambda_t}\left(v\cdot\nabla_xu,\,  h^\perp_\e\right).$$ 	
Noticing \eqref{hypoPoisson1} also implies that  
\begin{equation}\label{equw}
\left|\left(v\cdot\nabla_xu,\,  h^\perp_\e\right)\right|
\lesssim \|\la h_\e\ra-M_0\|_{L^2_x}\| h^\perp_\e\|_{L^2(\dif m)}
\le \| h_\e-M_0\|^2_{L^2(\dif m)}. 
\end{equation}
It means that the modified entropy $\E_\e$ is equivalent (independent of $\e$) to the square of the $L^2(\dif m)$-distance between $h_\e$ and $M_0$, when the constant $\delta>0$ is sufficiently small. 
		
Hence, we have 
\begin{equation*}
\frac{\dif}{\dif t} \E_\e(t) 
\lesssim -\left(\lambda_t+\lambda_t'\right) \E_\e(t). 
\end{equation*}
The conclusion \eqref{longtimeestimate} then follows from Grönwall's inequality and the equivalence between $\E_\e(t)$ and $\| h_\e(t,\cdot,\cdot)-M_0\|^2_{L^2(\dif m)}$. 
\end{proof}

We pointed out that the elliptic estimate \eqref{hypoPoisson1} for the Poisson equation \eqref{hypoPoisson0} used in the above proof resulting from a Poincaré type inequality essentially relies on the compactness of the spatial domain. 
It was shown in \cite{BDM} that the related elliptic estimate can be recovered by applying the Nash inequality \cite{Nash} when the spatial domain is the whole space $\R^d$, whose argument is under an abstract setting. 
Inspired by the proof of Proposition~\ref{longtime} above, we are also able to make the construction of \cite{BDM} precise to see that the argument still works for the nonlinear equation \eqref{epnk}. 
We remark that the following algebraic decay rate is optimal in the sense it is the same as in the linear case; see Appendix A of \cite{BDM}. 

\begin{proposition}\label{longwhole}
Assume that the initial data $\hine$ is valued in $[\lambda,\Lambda]$, and satisfies $\hine-M_1\in  L^1(\R^{2d},\dif m)$ for some universal constant $M_1>0$. 
Let the function $h_\e$ valued in $[\lambda,\Lambda]$ be a solution to \eqref{epnk} in $\R_+\times\R^{2d}$ associated with $h_\e|_{t=0}=\hine$. Then, for any $t>0$, 
\begin{equation*}
\|h_\e-M_1\|_{L^2(\R^{2d},\,\dif m)}
\lesssim \left(1+\|\hine-M_1\|_{L^1(\R^{2d},\,\dif m)}\right) t^{-\frac{d}{4}}. 
\end{equation*}
\end{proposition}

\begin{proof}
By the same derivation of \eqref{hypoestimate0} and \eqref{FPpert2} as in the proof of Proposition~\ref{longtime}, we have the microscopic coercivity
\begin{align}\label{hypo0w}
\frac{\dif}{\dif t}\| h_\e-M_1\|^2_{L^2(\R^{2d},\,\dif m)}
\lesssim -\frac{1}{\e^2}\| h^\perp_\e\|^2_{L^2(\R^{2d},\,\dif m)}, 
\end{align}
and the identity from the macro-micro decomposition that 
\begin{equation}\label{hypo1w}
\begin{split}
-\left(|v_1|^2\Delta_x w,\, \la h_\e\ra-M_0\right)_W &+\e\frac{\dif}{\dif t}\left(v\cdot\nabla_xw,\, h^\perp_\e\right)_W
=\left(v^{\otimes2}:\nabla_x^2w,\, h^\perp_\e\right)_W\\
&+\e\left(v\cdot\partial_t\nabla_xw,\, h^\perp_\e\right)_W-\frac{1}{\e}\left(\la h_\e\ra^\beta v\cdot\nabla_xw,\, h^\perp_\e\right)_W, 
\end{split}
\end{equation}
where $(\cdot,\cdot)_W$ denotes the $L^2(\R^{2d},\dif m)$ inner product, and the function $w(t,x)\in L_t^\infty([0,T];L_x^1\cap L_x^2(\R^d))$ is chosen to be the solution of the following ellipitc equation associated with the constant $\Theta:=\la |v_1|^2\ra$ and the macroscopic source $\la h_\e\ra-M_1$, 
\begin{equation}\label{poissonw}
w-\Theta\Delta_x w=\la h_\e\ra-M_1 {\quad\rm in\ }\R^d. 
\end{equation}
The elliptic estimate is derived by integrating \eqref{poissonw} against $-\Theta\Delta_x w$ so that 
\begin{equation}\label{ellipticw}
\begin{split}
\Theta\|\nabla_xw\|_{L_x^2(\R^d)}^2 + \Theta^2\|\nabla_x^2w\|_{L_x^2(\R^d)}^2
&= \left(-\Theta\Delta_x w,\, \la h_\e\ra-M_1\right)_W\\
&= \left(\la h_\e\ra-M_1 - w,\, \la h_\e\ra-M_1\right)_W=:\A. 
\end{split}
\end{equation}
It also follows from the same derivation as \eqref{hypoPoisson2} that 
\begin{equation*}
\|\nabla_x(\p_tw)\|_{L_x^2(\R^d)} \lesssim \frac{1}{\e}\|h^\perp_\e\|_{L^2(\R^{2d},\,\dif m)}. 
\end{equation*}
Gathering the above two estimates with \eqref{hypo1w}, we obtain 
\begin{equation*}
\begin{split}
\A + \e \frac{\dif}{\dif t}\left(v\cdot\nabla_xw,\, h^\perp_\e\right)_W
\lesssim \frac{1}{\e} \A^\frac{1}{2}\| h^\perp_\e\|_{L^2(\R^{2d},\,\dif m)}
+\|h^\perp_\e\|_{L^2(\R^{2d},\,\dif m)}^2, 
\end{split}
\end{equation*}
which implies from the Cauchy–Schwarz inequality that 
\begin{equation*}
\A + \e \frac{\dif}{\dif t}\left(v\cdot\nabla_xw,\, h^\perp_\e\right)_W
\lesssim \frac{1}{\e^2} \| h^\perp_\e\|_{L^2(\R^{2d},\,\dif m)}^2 . 
\end{equation*}
Denoting the modified entropy $\EE_\e(t):=\|h_\e-M_1\|_{L^2({\R^{2d},\,\dif m})}^2+\d\e\left(v\cdot\nabla_xw,\, h^\perp_\e\right)_W$ with a sufficient small constant $\d>0$, and using \eqref{hypo0w}, we conclude that  
\begin{equation}\label{eew}
\frac{\dif}{\dif t}\EE_\e(t)
\lesssim  -\A -\| h^\perp_\e\|_{L^2(\R^{2d},\,\dif m)}^2 . 
\end{equation}

Recalling the similar estimate \eqref{equw}, we see that $\EE_\e$ is equivalent to the square of the $L^2(\R^{2d},\dif m)$-distance between $h_\e$ and $M_1$. 
It thus suffices to recover $\la h_\e\ra-M_1$ by means of $\A$. 
By \eqref{poissonw} and the convexity of $|\cdot|$, we know that $|w|$ is a subsolution in the sense that 
\begin{equation*}
|w|-\Theta\Delta_x |w| \le |\la h_\e\ra-M_1| {\quad\rm in\ }\R^d, 
\end{equation*}
and hence $\|w\|_{L_x^1(\R^d)} \le \|\la h_\e\ra-M_1\|_{L_x^1(\R^d)}$.  
With the aid of Corollary~\ref{existwhole}, 
\begin{equation*}
\|w\|_{L_x^1(\R^d)} \le \|h-M_1\|_{L^1(\R^{2d},\,\dif m)}
\le \|\hine-M_1\|_{L^1(\R^{2d},\,\dif m)}. 
\end{equation*}
Applying \eqref{ellipticw} and the Nash inequality $\|w\|_{L_x^2(\R^d)}^{d+2}\lesssim \|w\|_{L_x^1(\R^d)}^2\|\nabla_x w\|_{L_x^2(\R^d)}^d$, we then acquire 
\begin{align*}
\|\la h_\e\ra-M_1\|_{L_x^2(\R^d)}^2
&\lesssim \A +\|w\|_{L_x^2(\R^d)}^2
\lesssim \A + \|\hine-M_1\|_{L^1(\R^{2d},\,\dif m)}^{\frac{4}{d+2}} \|\nabla_xw\|_{L_x^2(\R^d)}^{\frac{2d}{d+2}} \\
&\lesssim  \Big(\A^\frac{2}{d+2} + \|\hine-M_1\|_{L^1(\R^{2d},\,\dif m)}^{\frac{4}{d+2}}\Big)  \A^{\frac{d}{d+2}}. 
\end{align*}
Now that $\|h_\e-M_1\|_{L_x^2(\R^d)}^2\le (\Lambda+M_1)\|\hine-M_1\|_{L^1(\R^{2d},\,\dif m)}$, 
in the both cases of $\A + \|h^\perp_\e\|_{L^2(\R^{2d},\,\dif m)}^2 \lessgtr \|h_\e-M_1\|_{L^2(\R^{2d},\,\dif m)}^2$, we conclude that 
\begin{align*}
\|h_\e-M_1\|_{L^2(\R^{2d},\,\dif m)}^2
\lesssim \Big(1+\|\hine-M_1\|_{L^1(\R^{2d},\,\dif m)}^{\frac{4}{d+2}}\Big) 
(\A+\|h^\perp_\e\|_{L^2(\R^{2d},\,\dif m)}^2 )^{\frac{d}{d+2}}. 
\end{align*}
Combining this with \eqref{eew} and the equivalence between $\EE_\e$ and $\|h_\e-M_1\|_{L^2(\R^{2d},\,\dif m)}^2$, we have 
\begin{equation*}
\frac{\dif}{\dif t}\EE_\e(t) 
\lesssim -\Big(1+\|\hine-M_1\|_{L^1(\R^{2d},\,\dif m)}^{\frac{4}{d}}\Big)^{\!-1}
\EE_\e(t) ^{1+\frac{2}{d}}, 
\end{equation*}
Since $\EE_\e(0)\lesssim \|\hine-M_1\|_{L^1(\R^{2d},\,\dif m)}$, we arrive at 
\begin{align*}
\EE_\e(t)
&\lesssim \Big[\EE_\e(0)^{-\frac{2}{d}} + \Big(1+\|\hine-M_1\|_{L^1(\R^{2d},\,\dif m)}^{\frac{4}{d}}\Big)^{\!-1} t\Big]^{-\frac{d}{2}}\\
&\lesssim \Big(1+\|\hine-M_1\|_{L^1(\R^{2d},\,\dif m)}^2\Big) t^{-\frac{d}{2}}. 
\end{align*}
The proof is thus complete. 
\end{proof}

As far as the case $\e=1$ is concerned, we conclude the result of convergence to equilibrium.

\begin{proof}[Proof of Theorem \ref{wellpose} (ii)] 
Consider $g:=\mu^\frac{1}{2}h$. In view of Proposition~\ref{existence} and Corollary~\ref{existwhole} with the assumption on initial data, we know that $\lambda\le \mu^{-\frac{1}{2}}g\le\Lambda$ in $\R_+\times\O_x\times\R^d$, for $\O_x=\T^d$ or $\R^d$. 
By using Proposition~\ref{longtime} to $h=\mu^{-\frac{1}{2}}g$ with $\lambda_t=\lambda$ and $\O_x=\T^d$, we have an universal constant $c>0$ such that 
\begin{equation*}\label{gglong}
\big\| g(t)-M_0\m^\frac{1}{2}\big\|_{L^2(\T^d\times\R^d)}\lesssim e^{-2ct}.
\end{equation*}
Combining this with the Sobolev embedding and the interpolation, we derive that, for any $k\in\N$ with $k\ge d$, 
\begin{align*}
\big\|g(t)-M_0\mu^\frac{1}{2}\big\|_{{\mathcal C}^{k}(\T^d\times\R^d)}
&\lesssim_k \big\|g(t)-M_0\mu^\frac{1}{2}\big\|_{H^{2k}(\T^d\times\R^d)}\\
&\lesssim_k \big\|g(t)-M_0\mu^\frac{1}{2}\big\|_{H^{4k}(\T^d\times\R^d)}^\frac{1}{2}
\big\| g(t)-M_0\mu^\frac{1}{2}\big\|_{L^2(\T^d\times\R^d)}^\frac{1}{2}. 
\end{align*}
Since the $H^{4k}$-norm on the right hand side is bounded due to the global regularity estimate given by Proposition~\ref{smoothness}, we obtain the exponential convergence to equilibrium in each order derivative. 

The asserted result in the case $\O_x=\R^d$ is a direct consequence of Proposition~\ref{longwhole}. 
As a side remark, one is also able to upgrade the long time convergence to higher order derivatives of solutions by means of the global regularity estimate and interpolation as above; it yet gives the algebraic decay rate that is not optimal. 
\end{proof}

\subsection{Finite time asymptotics}\label{finitetimesection}
The study of macroscopic dynamics for the nonlinear kinetic model~\eqref{epnk} in this subsection relies on the regularity of the target equation~\eqref{limiteq}. 
On account of this, let us begin with mentioning some standard results for the equation~\eqref{limiteq} without proof. 
If the initial data satisfies $\lambda\le\r_{\rm in}\le\Lambda$, then such bounds are preserved along times, $\lambda\le\r\le\Lambda$, in the same spirit as Lemma~\ref{Gaussian}. Combining the parabolic De~Giorgi-Nash-Moser theory with the Schauder theory, we know that the solution $\r$ is smooth for any positive time. 
We state the a priori estimate precisely as follows, where its behavior near the initial time is taken into account in view of the standard scaling technique. 

\begin{lemma}\label{porouslemma}
Let $\rho_{\rm in}\in {\mathcal C}^{\a_0}(\T^d)$ valued in $[\lambda,\Lambda]$ with $\a_0\in(0,1)$, and $\r$ be the solution to \eqref{limiteq} in $\R_+\times\T^d$. 
Then, there is some universal constant $\a\in(0,1)$ such that 
\begin{equation}\label{porous0}
\left\|\r\right\|_{\mathcal{C}^\a(\R_+\times\T^d)}
\lesssim 1+\left\|\rho_{\rm in}\right\|_{{\mathcal C}^{\a_0}(\T^d)}.
\end{equation}
Moreover, there exists some constant $C_\r>0$ depending only on universal constants and $\left\|\rho_{\rm in}\right\|_{{\mathcal C}^{\a_0}(\T^d)}$, such that for any $t\in(0,1]$ and $x\in\T^d$, we have 
	\begin{equation}\label{porous}
	{t}^\frac{1-\a}{2}\left|\nabla_x\r(t,x)\right|
	+{t}^\frac{2-\a}{2}\left|\p_t\r(t,x)\right|
	+{t}^\frac{2-\a}{2}\left|\nabla^2_x\r(t,x)\right|
	+{t}^\frac{3-\a}{2}\left|\p_t\nabla_x\r(t,x)\right|\le C_\r; 
	\end{equation}
and for any $t\ge1$, we have 
	\begin{equation}\label{porous'}
	\left\|\nabla_x\r(t,\cdot)\right\|_{L^\infty(\T^d)}
	+\left\|\p_t\r(t,\cdot)\right\|_{L^\infty(\T^d)}
	+\left\|\nabla^2_x\r(t,\cdot)\right\|_{L^\infty(\T^d)}
	+\left\|\p_t\nabla_x\r(t,\cdot)\right\|_{L^\infty(\T^d)}\lesssim 1. 
	\end{equation}
\end{lemma}

We measure the distance between solutions to the scaled nonlinear kinetic model~\eqref{epfnk} and solutions to the fast diffusion equation~\eqref{limiteq} by the relative phi-entropy functional $\H_\b$ (see Definition~\ref{phi-entropy}). 
The following lemma shows the effectiveness of the relative phi-entropy for measuring $L^2$-distance, by virtue of the uniform convexity of $\varphi_\b$. 
It can be seen as a simple version of the Csiszár-Kullback inequality on the relative entropy. 
We give its statement below with a proof taken from \cite[Proposition 2.1]{DL} for the sake of completeness. 
\begin{lemma}\label{entropydist}
Let $h_1$ and $h_2$ be two functions valued in $[0,\Lambda]$. Then, we have
\begin{equation}\label{entropydist1}
		\H_\b(h_1|h_2)\ge \left(1-\frac{\b}{2}\right)\Lambda^{-\b} \|h_1-h_2\|^2_{L^2\left(\dif m\right)}. 
\end{equation}
If we additionally assume the lower bound that $ h_1,h_2\ge\lambda$, then 
\begin{equation*}\label{entropydist2}
		\H_\b(h_1|h_2)\le \left(1-\frac{\b}{2}\right)\lambda^{-\b} \|h_1-h_2\|^2_{L^2\left(\dif m\right)}. 
\end{equation*}
\end{lemma}
	
\begin{proof}
Since $\varphi_\b(1)=\varphi_\b'(1)=0$ and $\b\in[0,1]$, for any $z\in\R_+$, there exists $\xi_z\in\R_+$ lying between $1$ and $z$ so that 
\begin{equation*}
\varphi_\b(z)=\frac{1}{2}\varphi''_\b(\xi_z)(z-1)^2=\frac{2-\b}{2}\xi_z^{-\b}(z-1)^2. 
\end{equation*}
Since $\min\{z,1\}\le\xi_z\le\max\{z,1\}$, we have 
\begin{align*}
\int_{\T^d\times\R^d} \max\left\{h_1^{-\b},h_2^{-\b}\right\}|h_1-h_2|^2\dif m \le&\; \frac{2}{2-\b}\H_\b(h_1|h_2)\\
\le&\; \int_{\T^d\times\R^d} \min\left\{h_1^{-\b},h_2^{-\b}\right\}|h_1-h_2|^2\dif m,
\end{align*}
which implies the desired results by using the boundedness of $h_1,h_2$. 
\end{proof}

Let us consider the finite time diffusion asymptotics. 
\begin{proposition}\label{finitetime}
Let $\rho_{\rm in}\in {\mathcal C}^{\a_0}(\T^d)$ valued in $[\lambda,\Lambda]$ with $\a_0\in(0,1)$, and the sequence of functions $\left\{\hine\right\}_{\e\in(0,1)}\subset {\mathcal C}^{\a_0}(\T^d\times\R^d)$ satisfying 
		\begin{equation*}
		\la\hine\ra\ge\lambda{\ \ \rm in\ }\T^d {\quad\rm and\quad} 
		0\le\hine \le \Lambda {\ \ \rm in\ }\T^d\times\R^d. 
		\end{equation*}
Let $h_\e$ be the solutions of \eqref{epnk} associated with these initial data. 
Then, there exist some universal constants $\a\in(0,1)$, $C>0$, and some constant $C_\rho>0$ depending only on universal constants and $\left\|\rho_{\rm in}\right\|_{{\mathcal C}^{\a_0}(\T^d)}$, such that for any $\e\in(0,1)$ and for any $t\in[\underline{T},1]$ with $\underline{T}\in(0,1)$, the following estimate holds, 
\begin{align}\label{finitelimit}
\H_\b(h_\e|\rho)(t)
\le C_\r \H_\b(h_\e|\rho)(\underline{T})+C_\r \e \left({t}^{\frac{\a-1}{2}} +\e{t}^{\frac{\a-2}{2}}\right), 
\end{align}
where $\rho(t,x)$ is the solution to \eqref{limiteq} associated with the initial data~$\rho_{\rm in}$; 
and for any $t\ge1$, we have 
\begin{align}\label{finitelimit'}
\H_\b(h_\e|\rho)(t)
\le \left[\H_\b(h_\e|\rho)(t)(1)+ C\e\left(1+t^\frac{1}{2}\right)\right]e^{C t}. 
\end{align}
\end{proposition}

\begin{proof}
For $\b\in[0,1)$, the phi-entropy of $h_\e$ relative to $\rho$ reads 
\begin{align*}
\H_\b(h_\e|\rho)=\H_\b(h_\e|1)-\H_\b(\r|1)
-\frac{2-\b}{1-\b}\left(\la h_\e\ra-\r,\;\r^{1-\b}-1\right).
\end{align*}
As far as the entropy $\H_\b(h_\e|1)$ is concerned, the entropy dissipation is derived by the equation~\eqref{epnk}, integration by parts and using Hölder's inequality $\la\nabla_vh_\e\ra^2 \le \la h_\e\ra^{\b}\la h_\e^{-\b}|\nabla_vh_\e|^2\ra$, 
\begin{align}\label{entropydissipation}
\frac{\dif}{\dif t}\H_\b(h_\e|1)
&=\frac{2-\b}{1-\b}\left(h^{1-\b},\;h_t\right)
=-\frac{2-\b}{\e^2} \left(h_\e^{-\b}\nabla_vh_\e,\; \la h_\e\ra^\b\nabla_vh_\e\right)\nonumber\\
&\le-\frac{2-\b}{\e^2}\|\la\nabla_vh_\e\ra\|^2_{L^2_x}=-\frac{2-\b}{\e^2}\|\la vh_\e\ra\|^2_{L^2_x}.
\end{align}
In view of the limiting equation~\eqref{limiteq}, we have 
\begin{align}\label{entropy2}
\frac{\dif}{\dif t}\H_\b(\r|1)
=\frac{2-\b}{1-\b}\left(\r^{1-\b},\;\p_t\r\right)
=-(2-\b)\left(\r^{-\b}\nabla_x\r,\;\r^{-\b}\nabla_x\r\right).
\end{align}
A direct computation with the macroscopic equation~\eqref{hydroFP1} and the equation~\eqref{limiteq} leads to
\begin{align}\label{entropy3}
\frac{\dif}{\dif t}&\left(\la h_\e\ra-\r,\;\r^{1-\b}-1\right)
=\left(\r^{1-\b},\;\p_t\la h_\e\ra\right)
 +\left((1-\b)\r^{-\b}\la h_\e\ra-(2-\b)\r^{1-\b},\;\p_t\r\right)\nonumber\\
&=\frac{1-\b}{\e}\left(\r^{-\b}\nabla_x\r,\;\la vh_\e\ra\right)
 -\left((1-\b)\nabla_x\big(\r^{-\b}\la h_\e\ra\big)-(2-\b)\r^{-\b}\nabla_x\r,\;\r^{-\b}\nabla_x\r\right)
\end{align}
The evolution of $\H_\b(h_\e|\rho)$ is then estimated by combining \eqref{entropydissipation}, \eqref{entropy2} and \eqref{entropy3}, 
\begin{align*}
\frac{1}{2-\b}\frac{\dif}{\dif t}\H_\b(h_\e|\rho)
\le& -\frac{1}{\e^2}\|\la vh_\e\ra\|^2_{L^2_x}
  -\frac{1}{\e} \left(\la vh_\e\ra,\; \r^{-\b}\nabla_x\rho\right)
+\left(\nabla_x\big(\r^{-\b}\la h_\e\ra-\r^{1-\b}\big),\;\r^{-\b}\nabla_x\rho\right)\\
=&-\left\|\e^{-1}\la vh_\e\ra+\rho^{-\b}\nabla_x\r\right\|^2_{L^2_x}
   +\left(\e^{-1}\la vh_\e\ra,\;\rho^{-\b}\nabla_x\r\right)\\
&+\left(\nabla_x\la h_\e\ra+ 
\b\big(1-\r^{-1}\la h_\e\ra\big)\nabla_x\r,\;\rho^{-2\b}\nabla_x\rho\right)
\end{align*}
We remark that the above inequality also holds for $\b=1$ by a similar computation. 
Abbreviate $Q_\e:=\e^{-1}\la vh_\e\ra+\rho^{-\b}\nabla_x\rho$, $R_\e:=-\nabla_x\cdot\la v\otimes\nabla_vh_\e\ra-\e\partial_t\la vh_\e\ra$ and write the macroscopic equation \eqref{hydroFP2} in the form of
$\nabla_x\la h_\e\ra = -\e^{-1}\la h_\e\ra^{\b}\la vh_\e\ra + R_\e$. It then turns out that 
\begin{align}\label{retiveen0}
\frac{1}{2-\b}\frac{\dif}{\dif t}\H_\b(h_\e|\rho)
\le& -\|Q_\e\|^2_{L^2_{x}} 
+\left(\big(1-\r^{-\b}\la h_\e\ra^{\b}\big)Q_\e,\;\rho^{-\b}\nabla_x\rho\right)
+\left(R_\e,\;\rho^{-2\b}\nabla_x\rho\right)  \nonumber\\
&+\left(\r^{-\b}\la h_\e\ra^{\b}-\b\r^{-1}\la h_\e\ra -1+\b 
,\;\rho^{-2\b}\left|\nabla_x\r\right|^2\right)  \nonumber\\
\le&\;  2\left\|\rho^{-1-\b}|\nabla_x\rho|\right\|_{L^\infty_{t,x}}^2 \|\la h_\e\ra-\rho\|^2_{L^2_x}
+\left(R_\e,\;\rho^{-2\b}\nabla_x\rho\right),  
\end{align}
where for the second inequality, we used the Cauchy–Schwarz inequality 
\begin{equation*}
2\left(\big(1-\r^{-\b}\la h_\e\ra^{\b}\big)Q_\e,\;\rho^{-\b}\nabla_x\rho\right)
\le \|Q_\e\|^2_{L^2_{x}} + \left(\big|1-\r^{-\b}\la h_\e\ra^{\b}\big|^2,\;\big|\rho^{-\b}\nabla_x\rho\big|^2\right),
\end{equation*}
and the following two elementary inequalities, with $\b\in[0,1]$, 
\begin{equation*}
		\big|z^{\b}-1\big|\le|z-1| {\quad\rm and\quad} 
		\big|z^{\b}-\b z-1+\b\big|\le|z-1|^2, {\quad\rm for\ any\ } z\in\R_+. 
\end{equation*}
In view of Hölder's inequality and the inequality~\eqref{entropydist1} given in Lemma~\ref{entropydist}, we know that
\begin{align*}
\|\la h_\e\ra-\rho\|^2_{L^2_x}\le \|h_\e-\rho\|^2_{L^2(\dif m)}
\le 2\Lambda^{\b}\H_\b(h_\e|\rho). 
\end{align*}
Combining this with \eqref{retiveen0} and \eqref{porous}, we derive that, for any $t\in(0,1]$, 
\begin{align}\label{retiveen1}
\frac{\dif}{\dif t}\H_\b(h_\e|\rho)
\le C_\r{t}^{\a-1} \H_\b(h_\e|\rho)
	+2\left(R_\e,\;\rho^{-2\b}\nabla_x\rho\right), 
\end{align}
where the constants $\a\in(0,1)$ and $C_\r>0$ are provided by Lemma~\ref{porouslemma}. 

We point out that, after integrating in time, the remainder term in \eqref{retiveen1} involving $R_\e$ is of order $O(\e)$ due to the control of the entropy production and the regularity of the limiting equation. Indeed, 
\begin{align*}
\int_0^t\left(R_\e,\;\rho^{-2\b}\nabla_x\rho\right)
=& \int_0^t\left(\la v\otimes\nabla_vh_\e\ra,\;\nabla_x\big(\rho^{-2\b}\nabla_x\rho\big)\right)
  +\e\int_0^t\left(\la vh_\e\ra,\;\partial_t\big(\rho^{-2\b}\nabla_x\rho\big)\right) \\
&-\e\left(\la vh_\e\ra,\;\rho^{-2\b}\nabla_x\rho\right)(t) 
  +\e\left(\la vh_\e\ra,\;\rho^{-2\b}\nabla_x\rho\right)(0)\\
\lesssim&\; \left(\left\|\nabla_x\big(\rho^{-2\b}\nabla_x\rho\big)\right\|_{L^\infty_{t,x}} +\e\left\|\partial_t\big(\rho^{-2\b}\nabla_x\rho\big)\right\|_{L^\infty_{t,x}}\right) 
\int_0^t\left\|\la vh_\e\ra\right\|_{L^2_x}\\
&+\e\left\|\rho^{-2\b}\nabla_x\rho\right\|_{L^\infty_{t,x}}\|\la vh_\e\ra\|_{L^\infty\left([0,T];L^2_x\right)}. 
\end{align*}
It then follows from \eqref{porous}, \eqref{entropydissipation} and the global upper bound of $h_\e$ (Lemma~\ref{Gaussian}) that, for any $t\in(0,1]$, 
\begin{align*}
\int_0^t\left(R_\e,\rho^{-2\b}\nabla_x\rho\right)
\le&\; C_\r\left({t}^{\frac{\a-1}{2}}+\e{t}^{\frac{\a-2}{2}}\right)
\left(\int_0^t\left\|\la vh_\e\ra\right\|_{L^2_{x}}^2\right)^\frac{1}{2} 
+C_\r\e{t}^{\frac{\a-1}{2}} \nonumber\\
\le&\; C_\r\left(\e{t}^{\frac{\a-1}{2}} +\e^2{t}^{\frac{\a-2}{2}}\right) \sup\nolimits_{s\in[0,t]}\sqrt{\H_\b(h_\e|1)(s)} 
+C_\r\e{t}^{\frac{\a-1}{2}} \nonumber\\
\le&\; C_\r\left(\e{t}^{\frac{\a-1}{2}} +\e^2{t}^{\frac{\a-2}{2}}\right)
+C_\r\e{t}^{\frac{\a-1}{2}}. 
\end{align*}	
Combining this estimate with \eqref{retiveen1}, as well as Grönwall's inequality, we conclude \eqref{finitelimit}. 
Besides, we arrive at \eqref{finitelimit'}, if we apply Lemma~\ref{porouslemma} with \eqref{porous'} instead of \eqref{porous} in above argument. This completes the proof. 
\end{proof}

We are now in a position to conclude the global in time diffusion asymptotics. 
\begin{proof}[Proof of Theorem \ref{limit}]
We are going to combine Proposition~\ref{longtime}, \ref{finitetime} with a delicate analysis on the relative entropy around the initial time to get Theorem~\ref{limit}. 
The analysis is based on the barrier function method. 
Let us assume the constant $\a\in(0,1)$ provided by Proposition~\ref{finitetime}. 

\medskip\noindent{\textbf{Step 1.}}	Pointwise estimate. \\
Let us fix $\d\in(0,1)$, $(x_1,v_1)\in\T^d\times B_{R}$ with $R>0$, and consider the function 
\begin{equation*}
\overline{h}(t,x,v):=C_1t+C_2\left(|x-x_1-\e^{-1}tv|^2+|v-v_1|^2\right)
\end{equation*}
where the constants $C_1,C_2>0$ are to be determined.  
For any $t\le\frac{\e\d}{4(1+R)}$, we have
\begin{equation*}
\overline{h} \ge C_2\left(|x-x_1|^2+|v-v_1|^2-2\e^{-1}t|x-x_1||v|\right)\ge\frac{C_2\d^2}{2}=\Lambda
{\quad\rm on\ } \p B_\d(x_1,v_1), 
\end{equation*} 
where we chose $C_2:=2\d^{-2}\Lambda$. 
For any $(x,v)\in B_\d(x_1,v_1)$, 
\begin{equation*}
		|\la h\ra^\b \lfp\overline{h}|
		\lesssim |\Delta_v\overline{h}|+ |v\cdot\nabla_v\overline{h}|
		\lesssim \d^{-2}(1+R^2)\left(1+\e^{-2}t^2\right). 
\end{equation*}
Therefore, for any $t\le\frac{\e\d}{4(1+R)}$ and $(x,v)\in B_\d(x_1,v_1)$, 
\begin{equation*}
\left(\partial_t+\e^{-1}v\cdot\nabla_x-\e^{-2}\la h\ra^\b\lfp\right)\overline{h}
\ge C_1-C_0\e^{-2}\d^{-2}(1+R^2)\left(1+\e^{-2}t^2\right) \ge 0, 
\end{equation*}
where the constant $C_0>0$ is universal and we chose $C_1:=2C_0\e^{-2}\d^{-2}(1+R^2)$. 
Then, the maximum principle implies that, for any $t\le\frac{\e\d}{4(1+R)}$ and $(x,v)\in B_\d(x_1,v_1)$, 
\begin{equation}\label{coninitialbarrier}
|h_\e(t,x,v)-h_{\e,\rm in}(x_1,v_1)|
\le \overline{h}(t,x,v)+\sup\nolimits_{B_\d(x_1,v_1)}|\hine(x,v)-\hine(x_1,v_1)|. 
\end{equation}
In particular, for any $t\le\frac{\e\d}{4(1+R)}$ and $(x_1,v_1)\in\T^d\times B_{R}$, 
\begin{equation}\label{limitbarrier}
|h_\e(t,x_1,v_1)-h_{\e,\rm in}(x_1,v_1)|
\lesssim \e^{-2}\d^{-2}(1+R^2)t +\|h_{\e,\rm in}\|_{{\mathcal C}^{\a_0}(\T^d\times\R^d)}\d^\a. 
\end{equation}

As far as the solution $\r$ to the limiting equation~\eqref{limiteq} is concerned, using the H\"older estimate~\eqref{porous0} in Lemma~\ref{porouslemma}, we  derive that, for any $t\in\R_+$, 
\begin{align}\label{limitholder}
\|\r(t)-\r_{\rm in}\|_{L^\infty(\T^d)}
\lesssim \left(1+\|\r_{\rm in}\|_{{\mathcal C}^{\a_0}(\T^d)}\right)t^\a. 
\end{align}

\medskip\noindent{\textbf{Step 2.}}	Estimate of the relative entropy around the initial time. \\
Let us restrict our attention on the time interval $[0,\underline{T}]$ with $\underline{T}\in(0,1)$ to be determined. 
Compute the relative entropy $\H_\b(h_\e|\rho)$ in terms of initial data as follows, 
\begin{equation}\label{limitinitial0}
\begin{split}
\H_\b(h_\e|\rho)
  =\, \H_\b(\hine|\rho_{\rm in})
   +\int_{\T^d\times\R^d}\left[\left(\varphi_\b(h_\e)-\varphi_\b(\hine)\right)
   -\left(\varphi_\b(\r)-\varphi_\b(\r_{\rm in})\right)\right]\dif m&\\
-\int_{\T^d\times\R^d}\left[\left(\varphi'_\b(\r)-\varphi'_\b(\r_{\rm in})\right)(\hine-\r_{\rm in}) 
 +\varphi'_\b(\r)(h_\e-\hine) -\varphi'_\b(\r)(\r-\r_{\rm in})
 \right]\dif m&.
\end{split}
\end{equation}
Consider a truncation in $v$ for the integrals on the right hand side. 
Since $h_\e,h_{\e,\rm in},\r,\r_{\rm in}$ are all bounded from above (Lemma~\ref{Gaussian}), we have 
\begin{equation}\label{limitbarrierc}
\int_{\T^d\times B^c_R}\left[|\varphi_\b(h_\e)-\varphi_\b(\hine)|
+|\varphi_\b(\r)-\varphi_\b(\r_{\rm in})|\right]\dif m
\lesssim \int_{B^c_{R}}\dif\m\lesssim R^{-4}. 
\end{equation}
Observe that for any $a,b\in(0,\Lambda]$, there exists $\xi\in\R_+$ lying between $a$ and $b$ such that 
$\varphi_\b(a^2)-\varphi_\b(b^2)=2\xi\varphi'_\b(\xi^2)(a-b)$. 
Meanwhile, for any $\xi\in(0,\Lambda]$, $|\xi\varphi'_\b(\xi^2)|\lesssim1$. 
Thus, 
\begin{equation*}
|\varphi_\b(h_\e)-\varphi_\b(\hine)|\lesssim |h_\e-\hine|^\frac{1}{2}
{\quad\rm and\quad}
|\varphi_\b(\r)-\varphi_\b(\r_{\rm in})|\lesssim |\r-\r_{\rm in}|^\frac{1}{2}. 
\end{equation*}
Set $R:=\e^{-\frac{\eta}{4}}$, $\d:=\e^{\frac{\eta}{4}}$ and $\underline{T}:=\frac{1}{8}\e^{2+2\eta}$ for some constant $\eta\in(0,1)$ to be determined. 
In this setting, $\underline{T}\le\frac{\e\d}{4\la R\ra}$. 
It then follows from \eqref{limitbarrier}, \eqref{limitholder} and \eqref{limitbarrierc} that, for any $t\le\underline{T}$, 
\begin{align}\label{limitinitial1}
\int_{\T^d\times \R^d}&\left[\left|\varphi_\b(h_\e)-\varphi_\b(\hine)\right|
+\left|\varphi_\b(\r)-\varphi_\b(\r_{\rm in})\right|\right]\dif m\nonumber\\
&\ \lesssim R^{-4} +\int_{\T^d\times B_R}\left[|h_\e-\hine|^\frac{1}{2} +|\r-\r_{\rm in}|^\frac{1}{2}\right]\dif m\nonumber\\
&\ \lesssim \e^{\eta} + \e^{\frac{\eta}{2}}+\|h_{\e,\rm in}\|^\frac{1}{2}_{{\mathcal C}^{\a_0}(\T^d\times\R^d)}\e^{\frac{\a\eta}{8}}
+\Big(1+\|\r_{\rm in}\|^\frac{1}{2}_{{\mathcal C}^{\a_0}(\T^d)}\Big)\e^{\a+\a\eta}.  
\end{align}
Besides, $\|h_\e-\hine\|_{L^1(\T^d\times B^c_R,\,\dif m)}\lesssim R^{-4}$ and $|\varphi'_\b|\lesssim1$ on $[\lambda,\Lambda]$. 
Therefore, combining \eqref{limitbarrier}, \eqref{limitholder} with the inequality~\eqref{entropydist1} given in Lemma~\ref{entropydist} yields that, for any $t\le\underline{T}$, 
\begin{align}\label{limitinitial2}
\int_{\T^d\times\R^d}&\left[\left|\left(\varphi'_\b(\r)-\varphi'_\b(\r_{\rm in})\right)(\hine-\r_{\rm in})\right| 
 +\left|\varphi'_\b(\r)(h_\e-\hine)\right| +\left|\varphi'_\b(\r)(\r-\r_{\rm in})\right|
 \right]\dif m\nonumber\\
&\ \lesssim  \|\hine-\r_{\rm in}\|_{L^2\left(\T^d\times\R^d,\,\dif m\right)} 
+\|h_\e-\hine\|_{L^1\left(\T^d\times\R^d,\,\dif m\right)} +\|\r-\r_{\rm in}\|_{L^1(\T^d)}\nonumber\\
&\ \lesssim  \H_\b^\frac{1}{2}(\hine|\r_{\rm in})+ R^{-4} +\|h_\e-\hine\|_{L^1\left(\T^d\times B_R,\,\dif m\right)} +\|\r-\r_{\rm in}\|_{L^1(\T^d)}\nonumber\\
&\ \lesssim {\e'}^\frac{1}{2}
+\e^{\eta}+\|h_{\e,\rm in}\|_{{\mathcal C}^{\a_0}(\T^d\times\R^d)}\e^{\frac{\a\eta}{4}}
+\left(1+\|\r_{\rm in}\|_{{\mathcal C}^{\a_0}(\T^d)}\right)\e^{2\a+2\a\eta}. 
\end{align}
Plugging \eqref{limitinitial1} and \eqref{limitinitial2} into the expression \eqref{limitinitial0}, we derive that, for any $t\le\underline{T}$, 
\begin{equation}\label{limitinitial}
\H_\b(h_\e|\rho)(t)
\lesssim  
\left(1+\|h_{\e,\rm in}\|_{{\mathcal C}^{\a_0}(\T^d\times\R^d)}+\|\r_{\rm in}\|_{{\mathcal C}^{\a_0}(\T^d)}\right) (\e+\e')^{\frac{\a\eta}{8}}  .  
\end{equation}

\medskip\noindent{\textbf{Step 3.}}	Conclusion. \\
Recall that we have chosen $\underline{T}=\frac{1}{8}\e^{2+2\eta}$. 
In view of \eqref{limitinitial} and the estimate~\eqref{finitelimit} given in Lemma~\ref{finitetime}, one may optimize in $\eta$ to get the result. 
For simplicity, we pick $\eta:=\frac{\a}{4-2\a}$ so that $\underline{T}=\frac{1}{8}\e^{\frac{4-\a}{2-\a}}$ 
and $\e\Big(\underline{T}^{\frac{\a-1}{2}}+\e\underline{T}^{\frac{\a-2}{2}}\Big)\lesssim\e^\frac{\a}{2}$. 
It turns out that for any $t\in[0,1]$, 
\begin{align}\label{conclu}
\H_\b(h_\e|\rho)(t)
\le C_\r\H_\b(h_\e|\rho)(\underline{T})
+C_\r\e\left(\underline{T}^{\frac{\a-1}{2}}+\e\underline{T}^{\frac{\a-2}{2}}\right)
\le C_* (\e+\e')^{\frac{\a\eta}{8}},  
\end{align}
where the constant $C_\r>0$ is provided in Lemma~\ref{finitetime} and the constant $C_*>0$ depends only on universal constants, $\left\|\rho_{\rm in}\right\|_{{\mathcal C}^{\a_0}(\T^d)}$ and $\left\|\hine\right\|_{{\mathcal C}^{\a_0}(\T^d\times\R^d)}$.
Then, using the estimate \eqref{finitelimit'} given in Lemma~\ref{finitetime} with \eqref{entropydist1}, we arrive at point (i) of Theorem~\ref{limit}.

As for point (ii) of Theorem~\ref{limit}, applying \eqref{finitelimit'}, together with \eqref{entropydist1} and \eqref{conclu}, for any $t\in[1,\overline{T}]$, we have 
\begin{align*}
\|h_\e(t)-\rho(t)\|_{L^2(\dif m)}^2
\lesssim& \left[\H_\b(h_\e|\rho)(1)+ \e\Big(1+\overline{T}^\frac{1}{2}\Big)\right] e^{C\overline{T}}\\
\lesssim& \left[C_*(\e+\e')^{\frac{\a\eta}{8}} +\e\Big(1+(-\iota\log(\e+\e'))^\frac{1}{2}\Big)\right]
(\e+\e')^{-\frac{\a\eta}{16}}
\lesssim C_*(\e+\e')^{\frac{\a\eta}{16}}. 
\end{align*}
where we picked $\overline{T}:=-\iota\log(\e+\e')$ with $\iota:=\frac{\a\eta}{16C}$. 
Finally, using Proposition~\ref{longtime} under the additional assumption that $\hine\ge\lambda$, we know the long time behavior that there is some universal constant $c>0$ such that for any $t\ge\overline{T}$, 
\begin{equation*}
\|h_\e(t)-\rho(t)\|_{L^2(\dif m)}
\le \|h_\e(t)-M_0\|_{L^2(\dif m)}+\|\r(t)-M_0\|_{L_x^2}
\lesssim e^{-c\overline{T}}= (\e+\e')^{c\iota}.  
\end{equation*}
The proof now is complete.  
\end{proof}

\appendix
\section{Maximum principle}
The following maximum principle (on a not necessarily bounded domain) is repeatedly
applied throughout the article. We state it in the more suitable fashion for the Fokker-Planck equations of our concern, whose proof is in the same spirit as \cite[Lemma A.2]{CSS}. 
\begin{lemma}\label{max} 
	Let the domain $\o\subset\R^d\times\R^d$ and the parabolic cylinder $\o_T:=(0,T]\times\o$. 
	If $f\in\Ckin^2(\o_T)\cap{\mathcal C}^0(\overline{\o_T})$ is a bounded subsolution in the sense that 
\begin{equation}
\L_1f:=\left(\partial_t+v\cdot\nabla_{x}\right)f-\trace(AD_v^{2}f)-B\cdot\nabla_{v}f\le 0
{\quad\rm in\ }\o_T, 
\end{equation}
with the coefficients $A(t,x,v),B(t,x,v)\in{\mathcal C}^0({\o_T})$ satisfying 
\begin{equation*}
\lambda|\xi|^2\le A(t,x,v)\,\xi\cdot\xi\le\Lambda|\xi|^2,\quad
|B(t,x,v)\cdot\xi|\le\Lambda\la v\ra|\xi|
{\quad\rm\ for\ any\ }\xi\in\R^d,\ (t,x,v)\in\o_T, 
\end{equation*} 
then $\sup\nolimits_{\o_T} f\le\sup\nolimits_{\p_p\o_T}f$, 
where the parabolic boundary $\p_p\o_T:=[0,T]\times\overline{\o}-(0,T]\times\o$. 
\end{lemma}

\begin{proof}
If the domain $\o$ is bounded, then the result is classical. 
For general (unbounded) $\o$, we consider the auxiliary functions 
$\phi_1(t,v):=e^{C_1t}\la v\ra^2$ and $\phi_2(t,x):=e^{C_2t}\la x\ra^2$ with $C_1,C_2>0$. 
Since $f$ is bounded, for any $\varepsilon_1,\varepsilon_2>0$, there exists $R(\varepsilon_1),R(\varepsilon_2)>0$ (independent of $C_1,C_2$) such that $f-\varepsilon_1\phi_1-\varepsilon_2\phi_2\le\sup\nolimits_{\p_p\o_T}f$ in $\o_T\cap\{|x|\ge R(\varepsilon_2){\ \rm or\ }|v|\ge R(\varepsilon_1)\}$. 
	
By choosing $C_1=(d+2)\Lambda$, we have 
\begin{equation*}
\L_1\phi_1=e^{C_1t}(C_1\la v\ra^2-\trace(A)-2B\cdot v)
\ge (C_1-(d+2)\Lambda)\la v\ra^2= 0
{\quad\rm in\ }\o_T. 
\end{equation*}
For any $R_1\ge R(\varepsilon_1)$, there exists $C_2>0$ depending only on $R_1$ such that 
\begin{equation*}
\L_1\phi_2=e^{C_2t}\left(C_2\la x\ra^2 +v\cdot x\right)\ge (C_2-1)\la x\ra^2-|v|^2\ge 0
{\quad\rm in\ }\o_T\cap\{|v|<R_1\}. 
\end{equation*}
Therefore, for any $R_2>R(\varepsilon_2)$, $f-\varepsilon_1\phi_1-\varepsilon_2\phi_2$ is a subsolution to \eqref{linearequation} in the bounded domain $(0,T]\times(\w\cap(B_{R_2}\times B_{R_1}))$ with the data smaller than $\sup\nolimits_{\p_p\o_T}f$ on the boundary portion contained in $\{|x|=R_2{\ \rm or\ }|v|=R_1\}$. 
Then, applying the classical maximum principle yields 
\begin{equation*}
f-\varepsilon_1\phi_1-\varepsilon_2\phi_2 \le \sup\nolimits_{\p_p\o_{T}} f
{\quad\rm in\ }(0,T]\times(\w\cap(B_{R_2}\times B_{R_1})). 
\end{equation*}
Sending $R_2\rightarrow0$, $\varepsilon_2\rightarrow0$, $R_1\rightarrow0$, $\varepsilon_1\rightarrow0$ in order, we get the conclusion. 
\end{proof}

\section{Spreading of positivity}\label{lowerapp}
This appendix is devoted to the proof of Proposition~\ref{lower}. 
The argument follows the one presented in \cite{HST3} and it is based on the combination of Lemma~\ref{lowerlemma1} and Lemma~\ref{lowerlemma2}. 
\begin{proof}[Proof of Proposition~\ref{lower}]
	The proof is split into four steps. 
	
	\medskip\noindent{\textbf{Step 1.}}	 Spreading positivity for all velocities for short times. \\ 
	Applying Lemma~\ref{lowerlemma1} ($\tau=1$) yields that there is some universal constant $c_0>0$ such that, for any 
	$0\le t\le \min\big\{1,\,T,\,c_0\left\la r^{-1}\right\ra^{-2}\!\la v_0\ra^{-2}\big\}$, 
	\begin{equation*}
	h(t,x,v) \ge \frac{\delta}{8}\mathbbm{1}_{\left\{|x-x_0-tv|<\frac{r}{2},\,|v-v_0|<\frac{r}{2}\right\}}
	\ge\frac{\delta}{8}\mathbbm{1}_{\left\{|x-x_0-tv_0|<\frac{r}{4},\,|v-v_0|<\frac{r}{4}\right\}}. 
	\end{equation*}
	Let $r_0:=\min\big\{1,\frac{r}{16}\big\}$ and $\underline{t}:=\frac{\underline{T}}{2}$. Then, Lemma~\ref{lowerlemma2} implies that there exists $\underline{C}_{0}>0$ depending only on universal constants, $\underline{T}$, $\delta$, $r$ and $v_0$ such that, for any $0<\underline{t}\le t\le T_0$ with 
	$T_0:= \min\big\{1,\,T,\,c_0\left\la r^{-1}\right\ra^{-2}\!\la v_0\ra^{-2},\,\frac{r_0}{4}\la v_0\ra^{-1}\big\}$ and $v\in\R^d$, we have 
	\begin{equation}\label{lowers1}
	h(t,x,v)\ge \underline{C}_{0}^{-1} e^{-\underline{C}_{0}|v-v_0|^4}\mathbbm{1}_{\left\{|x-x_0-tv_0|<2r_0\right\}}
	\ge \underline{C}_{0}^{-1} e^{-\underline{C}_{0}|v-v_0|^4}\mathbbm{1}_{\left\{|x-x_0|<r_0\right\}}. 
	\end{equation}
	
	\medskip\noindent{\textbf{Step 2.}}	 Spreading positivity in space for short times. \\ 
	For any fixed $\overline{t}\in[\underline{t},T_0]$ and $\overline{x}\in\T^d$, we set $\overline{v}:=\frac{\overline{x}-x_0}{\overline{t}-\underline{t}}$. 
	In view of \eqref{lowers1}, by Lemma~\ref{lowerlemma1} ($\tau=2(\overline{t}-\underline{t})$, $v_0=\overline{v}$), we deduce that, 
	if $\overline{t}-\underline{t}\le c_0
	\left\la 2(\overline{t}-\underline{t})r_0^{-1}\right\ra^{-2}\la\overline{v}\ra^{-2}$, 
	in particular if  
	\begin{equation}\label{lowers2t}
	\overline{t}\le\underline{t}+\overline{t}_0 {\ \ \rm with\ } \overline{t}_0:=\frac{c_0r_0^2}{4+r_0^2}\la\overline{x}-x_0\ra^{-2},
	\end{equation}
	then there exists $\delta_0>0$ with the same dependence as $\underline{C}_{0}$ such that, for any $t\in[\underline{t},\overline{t}]$, 
	\begin{equation*}
	h(t,x,v)
	\ge \delta_0\mathbbm{1}_{\left\{|x-x_0-({t}-\underline{t})v|<\frac{r_0}{2},\;
		2(\overline{t}-\underline{t})|v-\overline{v}|<\frac{r_0}{2}\right\}}
	\ge \delta_0\mathbbm{1}_{\left\{|x-x_0-({t}-\underline{t})\overline{v}|<\frac{r_0}{4},\;
		|v-\overline{v}|<\frac{r_0}{4}\right\}}. 
	\end{equation*}
	Then, Lemma~\ref{lowerlemma2} ($v_0=\overline{v}$) implies that, 
	for any $0<2\underline{t}\le t\le\underline{t}+\overline{t}_0$ and $v\in\R^d$,
	\begin{equation}\label{lowers21}
	h(t,x,v)
	\ge \underline{C}_{1}^{-1}e^{-\underline{C}_{1}|v|^4} \mathbbm{1}_{\left\{|x-x_0-({t}-\underline{t})\overline{v}|<\frac{r_0}{8}\right\}},
	\end{equation}
	for some constant $\underline{C}_{1}>0$ depending only on universal constants, $\underline{T},\delta,r,v_0$ and $|\overline{x}-x_0|$. 
	In particular, for any $0<2\underline{t}\le \overline{t}\le\underline{t}+\overline{t}_0$ and $v\in\R^d$,
	\begin{equation}\label{lowers22}
	h(\overline{t},\overline{x},v)\ge \underline{C}_{1}^{-1}e^{-\underline{C}_{1}|v|^4}. 
	\end{equation}
	
	\medskip\noindent{\textbf{Step 3.}}	 Spreading positivity for any finite time. \\
	We observe that the time interval above is restricted (see \eqref{lowers2t}), but it can be removed by applying the lemmas again. 
	Based on the previous step, it suffices to deal with the case that $\overline{t}>\overline{t}_0$. 
	By a similar proof to \eqref{lowers21}, we derive 
	\begin{equation*}
	h(\overline{t}_0,x,v)\ge \delta_1\mathbbm{1}_{\left\{|x-\overline{x}|<\frac{r_0}{8},\,
		|v|<\frac{r_0}{8}\right\}}, 
	\end{equation*}
	for some constant $\delta_1>0$ with the same dependence as $\underline{C}_{1}$. 
	In view of this data, applying Lemma~\ref{lowerlemma1} to $h(\overline{t}_0+\cdot,\cdot,\cdot)$ (with $\tau=1$, $v_0=0$), we see that, for any $t\in\big[\overline{t}_0,\,\min\big\{T_0,\overline{t}_0+T_1\big\}\big]$ with $T_1:=c_0\left\la\frac{8}{r_0}\right\ra^{-2}$,
	\begin{equation*}
	h(t,x,v)\ge \frac{\delta_1}{8}\mathbbm{1}_{\left\{|x-\overline{x}|<\frac{r_0}{16},\, 
		|v|<\frac{r_0}{16}\right\}}. 
	\end{equation*}
	It then follows from Lemma~\ref{lowerlemma2} that, for any $t\in\big[\overline{t}_0+\underline{t},\,\min\big\{T_0,\overline{t}_0+T_1\big\}\big]$ and $v\in\R^d$,  
	\begin{equation*}
	h(t,x,v)\ge \underline{C}_{2}^{-1} e^{-\underline{C}_{2}|v|^4}\mathbbm{1}_{\left\{|x-\overline{x}|<\frac{r_0}{32}\right\}}, 
	\end{equation*}
	for some constant $\underline{C}_{2}>0$ with the same dependence as $\underline{C}_{1}$.  
	
	Combining this with \eqref{lowers22}, as well as recalling that $\underline{T}=2\underline{t}$ and the space domain $\T^d$ is compact, we know that there exists $\underline{C}_{3}>0$ depending only on universal constants, $\underline{T},\delta,r$ and $v_0$ such that, for any $(t,x,v)\in\left[\underline{T},\,\min\left\{T_0,T_1\right\}\right]\times\T^d\times\R^d$, 
	\begin{equation*}
	h(t,x,v)\ge\underline{C}_{3}^{-1} e^{-\underline{C}_{3}|v|^4}.
	\end{equation*}
	Since $T_0$ and $T_1$ depend only on universal constants, $r$ and $v_0$, by applying the above arguments iterately, we obtain the result for any finite time. 
	
\medskip\noindent{\textbf{Step 4.}}	 Improving the exponential tail. \\
We remark that this step is not necessary for the applications of the lower bound result, but it shows a more precise decay rate as $|v|\rightarrow\infty$. 
	
By the previous step, there is some $\underline{c}>0$ depending only on universal constants, $\underline{T},T,\delta,r$ and $v_0$ such that $h\ge \underline{c}$ in $[\underline{T},T]\times\T^d\times B_1$. 
Consider the barrier function  
$$\underline{h}(t,x,v):=\underline{c} e^{-C_0(t-\underline{T})^{-1}|v|^2}{\quad\rm in\ }[\underline{T},T]\times\T^d\times B_1^c,$$
where the constant $C_0>1$ is to be determined. 
By recalling the equation~\eqref{rnk} and a direct computation, we have 
	\begin{align*}
	\left(\partial_t+v\cdot\nabla_x\right)\underline{h} -\rfp_h\,\lfp\underline{h} 
	&= \frac{C_0\rfp_h\underline{h}}{(t-\underline{T})^2}
	\left(\rfp_h^{-1}+2(d-|v|^2)(t-\underline{T})-4C_0|v|^2\right) \\
	&\le \frac{C_0\rfp_h\underline{h}}{(t-\underline{T})^2} 
	\left(\underline{c}^{-\b}+2dT-4C_0\right)
	{\quad \rm in\ }
	(\underline{T},T]\times\T^d\times B_1^c.
	\end{align*}
	In particular, by choosing $C_0$ sufficiently large (with the same dependence as $\underline{c}$), we have
	\begin{align*}
	\left(\partial_t+v\cdot\nabla_x\right)\left(\underline{h}-h\right) -\rfp_h\,\lfp\left(\underline{h}-h\right)
	\le 0
	{\quad \rm in\ }
	(\underline{T},T]\times\T^d\times B_1^c. 
	\end{align*} 
	Besides, by its definition, $h\ge\underline{h}$ on the boundary $\{t\in[\underline{T},T],\,|v|=1\}\cup \{t=2\underline{t},\,|v|\ge1\}$. 
	The maximum principle (Lemma~\ref{max}) then implies that $h\ge\underline{h}$ in $[\underline{T},T]\times\T^d\times B_1^c$. 
	Therefore, we achieve the Gaussian type lower bound for any $(t,x,v)\in[2\underline{T},T]\times\T^d\times\R^d$. 
	The proof now is complete. 
\end{proof}

\section{Gaining regularity of spatial increment}
This appendix is devoted to the proof of two technical lemmas for spatial increments involved in the bootstrapping of higher regularity for solutions to the equation~\eqref{gnk} presented in Subsection~\ref{smoothsection}. 
For the convenience of the reader, we report a brief proof following the lines of 
\cite[Lemma 8.1]{IS} with $s=1$, $\a_1=\b=2$. 

\begin{lemma}\label{mio-lemma}
Let $\a\in(0,1)$ and a bounded continuous function $g$ defined in $Q_4$. 
If there exists some constant $M>0$ such that for any $y \in B_1$, 
\begin{equation*}\label{assumptions1}
	[\delta_{y}g]_{\Ck^{0}(Q_2)} \le M{\quad\rm and\quad} [ \delta_{y}g ]_{\Ck^{2+\a}(Q_2)} \le M \|(0,y,0)\|^{2}, 
\end{equation*}
then there exists some universal constant $\eta\in(0,1)$ such that for any $y \in B_1$, 
\begin{equation*}
	\| \delta_{y}g \|_{\Ck^{\eta}(Q_1)} \lesssim M \| (0,y,0) \|^{3}.
\end{equation*}
\end{lemma}

\begin{proof}[Proof of Lemma~\ref{mio-lemma}]
In view of the assumption and Remark~\ref{holderremark}, for fixed $y\in B_1$, we consider the the polynomial expansion $p_0$ of $\d_yg$ at $z_0\in Q_2$ with $\dk(p_0)=2$, 
\begin{equation*}
p_0(z)= \d_yg(z_0) +(\p_t+v_0\cdot\nabla_x)\d_yg(z_0)\,t +\nabla_v\d_yg(z_0)\cdot v + \frac{1}{2} D^2_v\d_yg(z_0)\,v\cdot v, 
\end{equation*} 
for $z:=(t,x,v)\in\R\times\R^d\times\R^d$. 
For any $z$ such that $z_0\circ z\in Q_4$, we have 
\begin{equation}\label{assumptions2}
|\d_yg(z_0\circ z)- p_0(z)|\le M \|(0,y,0)\|^2\|z\|^{2+\a}. 
\end{equation}
In particular, $p_0(0,y,0)=\d_yg(z_0)$ so that for any $y\in B_1$,
\begin{align*}
| \d_{2y}g(z_0) - 2\d_yg(z_0)| 
&= |\d_yg(z_0\circ (0,y,0)) - \d_yg(z_0)| 
= |\d_yg(z_0\circ (0,y,0)) - p_0(0,y,0)| \nonumber\\ 
&\le M \| (0,y,0) \|^{4+\a}
\end{align*}
It then follows that for any $z_0\in Q_2$ and for any $k \in \N$ such that $z_0\circ(0,2^ky,0)\in Q_4$, 
\begin{align}\label{gaining1}
\big|\d_yg(z_0) - 2^{-k}\d_{2^{k}y}g(z_0)\big| 
&\le \sum_{j=1}^{k}2^{-j}\big|\d_{2^{j}y}g(z_0)-2\d_{2^{j-1}y}g(z_0)\big|\nonumber\\
&\le M\|(0,y,0)\|^{4+\a} \sum_{j=1}^{k}2^{\frac{(1+\a)j}{3}} 
\le 2M \|(0,y,0)\|^{4+\a} 2^{\frac{(1+\a)k}{3}}.
\end{align}
Picking $k\in\N$ such that $\|2^{k-1}(0,y,0)\|\le1<\|2^{k}(0,y,0)\|$ and using the assumption yields
\begin{align}\label{gaining2}
|\d_yg(z_0)| &\le 2^{-k}|\d_{2^ky}g(z_0)|+2M\|(0,y,0)\|^{4+\a} 
2^{\frac{(1+\a)k}{3}} \nonumber\\
&\le \|\d_{2^ky}g\|_{\Ck^0(Q_2)}\|(0,y,0)\|^{3}+4M\|(0,y,0)\|^{3}
\le 5M\| (0,y,0) \|^{3}.
\end{align}

It remains to show that there exists some constant $\eta>0$ depending only on $\a$ such that 
\begin{equation}\label{toprove}
|\d_yg(z_0\circ z)-\d_yg(z_0)|\lesssim M\| (0,y,0) \|^{3} \|z\|^{\eta}. 
\end{equation}
By \eqref{assumptions2} and Lemma~\ref{kinetic-degree}, we know that for any $z_0\in Q_1$ and $z_0\circ z\in Q_4$, 
\begin{align*}
|\d_yg(z_0\circ z)-\d_yg(z_0)|
\le& \left(|(\p_t+v_0\cdot\nabla_x)\d_yg(z_0)|
+|D^2_v\d_yg(z_0)|\right)\|z\|^2\\
 &+|\nabla_v\d_yg(z_0)|\|z\| +M\|(0,y,0)\|^2\|z\|^{2+\a}\\
\lesssim& \Big([\d_yg]_{\Ck^{2+\a}(Q_2)}\|z\| +[\d_yg]_{\Ck^{2+\a}(Q_2)}^\frac{1}{2}[\delta_{y}g]_{\Ck^0(Q_2)}^\frac{1}{2}+[\delta_{y}g]_{\Ck^0(Q_2)}\Big)\|z\|\\
 &+M\|(0,y,0)\|^2\|z\|^{2+\a}. 
\end{align*}
If $\|z\|\le\|(0,y,0)\|$, then combining the above expression with the assumption and \eqref{gaining2} implies \eqref{toprove} with $\eta=\frac{1}{2}$. 
In particular, if $k\in\N$ such that $\|z\|<\|2^{k}(0,y,0)\|$, then we have 
\begin{equation}\label{toprove1}
2^{-k}|\d_{2^ky}g(z_0\circ z)-\d_{2^ky}g_{z}(z_0)| 
\lesssim 2^{-k}M\|(0,2^ky,0)\|^3 \|z\|^\eta
= M\|(0,y,0)\|^3 \|z\|^\eta. 
\end{equation}
Now if $\|z\|\ge\|(0,y,0)\|$, 
applying \eqref{gaining1} at points $z_0$ and $z_0\circ z$, with $k\in\N$ such that $\|2^{k-1}(0,y,0)\|\le\|z\|<\|2^{k}(0,y,0)\|$, yields that 
\begin{align}
|\d_yg(z_0) - 2^{-k}\d_{2^ky}g(z_0)| 
\le 4M\|(0,y,0)\|^3\|z\|^{1+\a}, \label{toprove2}\\
|\d_yg(z_0\circ z) - 2^{-k}\d_{2^ky}g(z_0\circ z)| 
\le 4M \|(0,y,0)\|^3\|z\|^{1+\a}.\label{toprove3}
\end{align}
Summing up \eqref{toprove1}, \eqref{toprove2} and \eqref{toprove3}, we arrive at \eqref{toprove}. 
\end{proof}

Following the lines of the above proof and taking into account that $\|g\|_{\Ck^{2+\a}(Q_2)}\le M$, one is also able to prove the following result. 
\begin{lemma}\label{incrementi-IS}
If $g\in \Ck^{2+\a}(Q_2)$ with $\a\in(0,1)$, then for any $y \in B_1$, we have 
\begin{equation*}
\|\delta_yg\|_{\Ck^\a(Q_1)} \lesssim \|g\|_{\Ck^{2+\a}(Q_2)}\|(0,y,0)\|^{2}. 
\end{equation*}
\end{lemma}

	\bibliographystyle{plain}
	\bibliography{AZbib}
\end{document}